\documentclass[a4paper,12pt,reqno]{amsart}

\usepackage{microtype}
\usepackage{amssymb,amsthm,amsmath,mathrsfs}
\usepackage{mathtools}
\usepackage{bbm}
\usepackage{stackengine}
\usepackage{enumitem}
\usepackage{soul}
\usepackage{thm-restate}

\usepackage[foot]{amsaddr}

\usepackage{caption,color,graphicx}
\usepackage[dvipsnames]{xcolor}
\usepackage{soul}

\usepackage{hyperref}
\hypersetup{colorlinks  = true,
            urlcolor    = blue,
            linkcolor   = red,
            citecolor   = blue}

\usepackage{verbatim}

\usepackage{environ}
\mathtoolsset{showonlyrefs,showmanualtags}

\usepackage[centering, top=3cm, bottom=2.5cm]{geometry}

\allowdisplaybreaks

\newcommand{\N}{\mathbb{N}}
\newcommand{\Z}{\mathbb{Z}}

\newcommand{\R}{\mathbb{R}}
\newcommand{\C}{\mathbb{C}}

\renewcommand{\d}{\operatorname{d}\!}

\newcommand{\e}{\varepsilon}

\newcommand{\s}{\operatorname{span}}

\newcommand{\vertiii}[1]{{\left\vert\kern-0.25ex\left\vert\kern-0.25ex\left\vert #1 \right\vert\kern-0.25ex\right\vert\kern-0.25ex\right\vert}}

\newcommand{\ex}[1]{e^{-2\pi i #1}}

\newcommand{\f}[2]{f_{#1}^{(#2)}}

\newcommand{\df}[2]{\dot{f}_{#1}^{(#2)}}
\let\oldtocsection=\tocsection
\let\oldtocsubsection=\tocsubsection
\renewcommand{\tocsection}[2]{\hspace{0em}\oldtocsection{#1}{#2}}
\renewcommand{\tocsubsection}[2]{\hspace{1em}\oldtocsubsection{#1}{#2}}

\theoremstyle{plain}
\newtheorem*{theorem*}{Theorem}
\newtheorem{theorem}{Theorem}[section]
\newtheorem{lemma}[theorem]{Lemma}
\newtheorem{proposition}[theorem]{Proposition}

\theoremstyle{definition}
\newtheorem{definition}[theorem]{Definition}
\newtheorem{remark}[theorem]{Remark}
\newtheorem{example}[theorem]{Example}

\newtheorem{manualtheorem}{Theorem}

\newtheorem{manualcorollary}{Corollary}

\newtheorem{manualstep}{\bf Step}

\newtheorem{manualhyp}{Hypothesis}

\IfPackageLoadedTF{MnSymbol}{}{}

\author{Matheus M. Castro and Gary Froyland}
\address{School of Mathematics and Statistics, UNSW Sydney, Sydney NSW 2052, Australia}
\email{m.manzatto\_de\_castro@unsw.edu.au, g.froyland@unsw.edu.au}
\usepackage{nicematrix}
\usepackage{graphicx}
\newcommand\smallO{
  \mathchoice
    {{\scriptstyle\mathcal{O}}}
    {{\scriptstyle\mathcal{O}}}
    {{\scriptscriptstyle\mathcal{O}}}
    {\scalebox{.7}{$\scriptscriptstyle\mathcal{O}$}}
  }
\usepackage[dvipsnames]{xcolor}

\newcommand*\rev[1]{#1}
\usepackage{tikz}
\usetikzlibrary{arrows.meta,calc,positioning,decorations.markings,fit}
\usetikzlibrary{backgrounds}

\newcommand*\mmc[1]{{\color{PineGreen}{#1}}}

\usepackage{accents}
\newlength{\dhatheight}
\newcommand{\doublehat}[1]{%
    \settoheight{\dhatheight}{\ensuremath{\hat{#1}}}%
    \addtolength{\dhatheight}{-0.35ex}%
    \hat{\vphantom{\rule{1pt}{\dhatheight}}%
    \smash{\hat{#1}}}}

\definecolor{darkorange}{rgb}{0.8, 0.33, 0.0}

\usepackage{fullpage}
\usepackage{tikz}
\title[Notes on the cycles problem]{On the structure of complex spectra and eigenfunctions of transfer and Koopman operators} 
\date{\today}

\begin{document}

\begin{abstract}
Complex eigenspectra of transfer and Koopman operators describe rotational motion in dynamical systems.
A particularly relevant situation in applications is when the rotation speed depends on the state-space position of the dynamics.
We consider a canonical model of such dynamics in the presence of small noise, and provide precise characterisations of the eigenspectrum and eigenfunctions of the corresponding transfer operators.
Further, we study the limiting behaviour of the eigenspectrum and eigenfunctions in the zero-noise limit, including their quadratic and linear response.
Our results clarify the structure of transfer and Koopman operator eigenspectra, and provide new interpretations relevant to applications.
Our theorems on support localisation of the eigenfunctions yield simple algorithms to detect the existence and state-space location of approximately cyclic motion with distinct periods.
Our \rev{numerical results verify} that information on the cycle periods and their locations determined by the operator eigendata is \rev{insensitive to noise level in the linear response regime.} 
We believe that the dynamic mechanisms underlying the eigendata and their properties apply rather broadly and enhance our understanding of approximate cycle detection in dynamical systems with operator methods.

\end{abstract}


\maketitle

\section{Introduction}\label{sec:1}
Over the last three decades, operator-theoretic techniques have proven to be highly successful methods of analysis for complex dynamical systems.
The basis of the analysis is usually the dominant eigenspectrum of a transfer operator or Koopman operator and the corresponding eigenfunctions.
In the autonomous setting, \rev{where the underlying dynamics does not change in time,} large real spectral values correspond to almost-invariant sets \rev{(large regions of state space that have no collective state-space motion under the dynamics)} \cite{DellnitzJunge99,dellnitz2000theisolated}.
This concept has been applied extensively in several fields, including molecular dynamics \cite{schutte2001transfer} \rev{to find stable molecular conformations}, space mission design \cite{dellnitz2006space} \rev{to identify orbital trapping regions}, and physical oceanography \cite{froyland2007detection,froyland2014well,miron2019} \rev{to map ocean gyres and garbage patches}.
Regarding \textit{complex} eigenvalues, the Halmos--von Neumann Theorem (see e.g. \cite[Theorem 17.11]{eisner2015operator}) states that each ergodic measure-preserving system with discrete spectrum is isomorphic to an ergodic rotation system, \rev{for example, an aperiodic rotation on a $n$-torus}.
This concept has been used extensively in the context of Koopman operators, e.g.\ \cite{mezic2004comparison,rowley2009spectral}.

The current paper takes the underlying Halmos--von Neumann idea and extends it in two ways.
Firstly, \textit{we allow the rotation to depend on the location in the state space}, rather than being fixed along a particular coordinate throughout the state space. Secondly, \textit{we make a detailed study of the effect of noise} on the complex eigenvalues and eigenfunctions.
The motivation for this study is that complex dynamical systems often possess many co-existing cycles or approximate cycles. 
For example, in the Earth system, climate drivers such as the El-Ni\~no Southern Oscillation (ENSO)
and the Madden--Julian Oscillation (\rev{MJO}) are approximate co-existing cycles that influence the weather and climate.
On even longer timescales, changing glacial cycles describe how northern-hemisphere continental ice sheets repeatedly accumulate and melt over millions of years.

Recent work \cite{FGLPS21,FGLS24} has demonstrated that transfer operator methods are a flexible and reliable way of extracting cycles from models, observations, and even time series of spatial fields;  for example extracting the ENSO cycle from a time series of observed sea-surface temperature fields.
Long-lived persistent cycles typically give rise to complex eigenvalues in the dominant spectrum of the transfer operator.
\textit{The magnitude of the eigenvalue quantifies the rate of temporal decay of the cycle, while the argument of the eigenvalue directly provides an estimate of the cycle's period.
Moreover, the corresponding eigenfunctions may be used to represent the cycle in terms of the original spatiotemporal time series \cite{FGLPS21}}.

These ideas may be illustrated with the Lorenz equations, where an estimate of the transfer operator is constructed and the complex eigenvalue with largest magnitude is computed. 
The left panel of Figure \ref{fig:lorenz} displays the magnitude of the complex entries of the corresponding eigenvector on a blue colourscale, and the right panel shows the argument of the complex entries on a periodic colourscale.
\begin{figure}[h!]
    \centering
\includegraphics[width=0.49\linewidth]{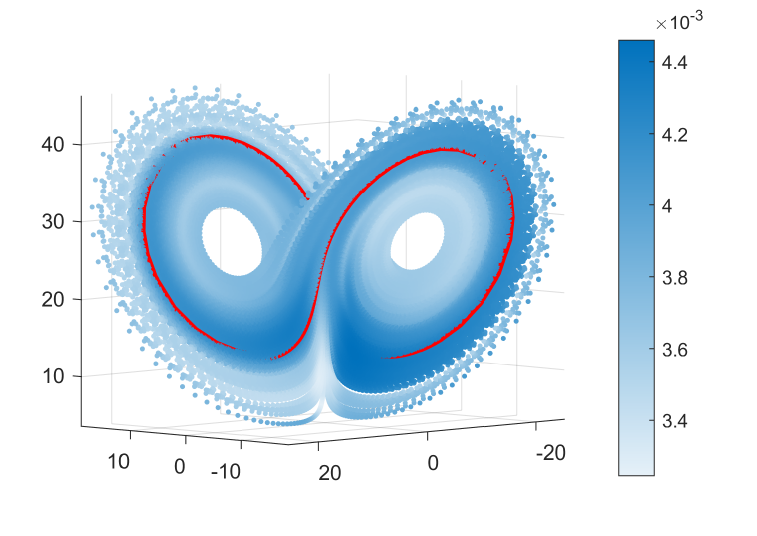}
\includegraphics[width=0.49\linewidth]{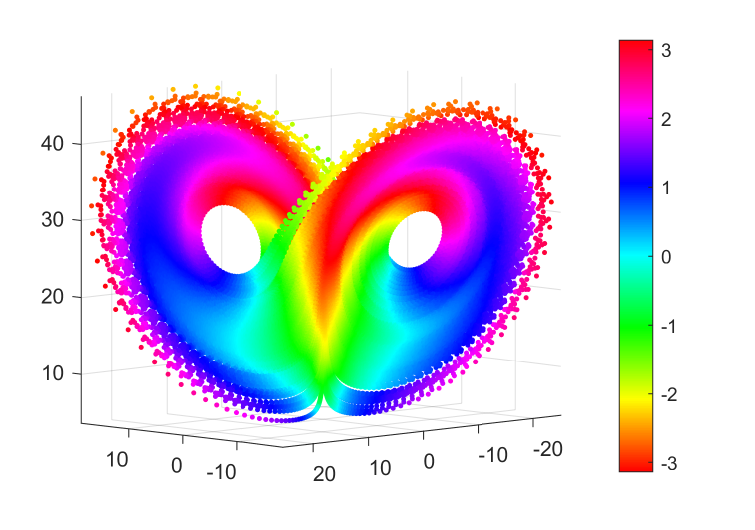}
    \caption{Left: Magnitude of the leading complex eigenvector of the normalised transfer operator for the Lorenz flow. A peak in magnitude occurs in the vicinity of the lowest-period periodic orbit of the Lorenz equations with a period of $1.5586$ time units (shown in red, see Fig.\ 10a \cite{FP09} and \cite{FGZ93}). 
    Right: The argument of the complex eigenvector. In the periodic colormap, one sees a full rotation around each of the two wings as the argument proceeds from $-\pi$ to $\pi$;  see \cite{FGLPS21}.}
    \label{fig:lorenz} 
\end{figure}
In this example, the argument of the eigenvalue associated with this complex eigenvector corresponds to a cycle of period $0.770$.
This period is roughly half that of the lowest-period periodic orbit of the Lorenz equations, which loops once around each wing (displayed in red in Figure \ref{fig:lorenz}(left)).
We make this comparison to emphasise that we are \emph{not} aiming to detect single periodic orbits, but instead to capture large regions in the state space that revolve with an approximately constant period.
The eigenvector detects that the dominant cycle in the Lorenz system consists of wide bands on each wing that rotate around each wing with a period of roughly $0.770$.
The symmetry of the Lorenz equations under $(x,y,z)\mapsto (-x,-y,z)$, is reflected in the representation of the cycle in Figure \ref{fig:lorenz}(right): \rev{the cycle phase given by the argument respects this symmetry}, see \cite{FGLPS21} for details.

Our focus in this  work is to provide a formal explanation for the behaviour of the spectrum and eigenfunctions of the transfer operator in the presence of coexisting cycles with different periods.
Roughly speaking, \textit{eigenfunctions corresponding to cycles with distinct periods should have supports that are largely disjoint}, because large overlaps would need to simultaneously rotate at two distinct speeds.
Our formal analysis concentrates on a mathematical model with multiple coexisting cycles, acting on a discrete collection of circles, between which trajectories may randomly jump under the influence of noise.
Our state space $M$ is the union of $N$ circles:  $M:=\{1,\ldots,N\}\times \mathbb{S}^1$, which we view as a cylinder that has been discretised along its central axis.
The dynamics on the $j^{\rm th}$ circular band of the discretised cylinder is a rotation with speed $\alpha_j$;  \rev{see Figure \ref{fig:model}}.
\begin{figure}
    \centering
\includegraphics[width=0.5\linewidth]{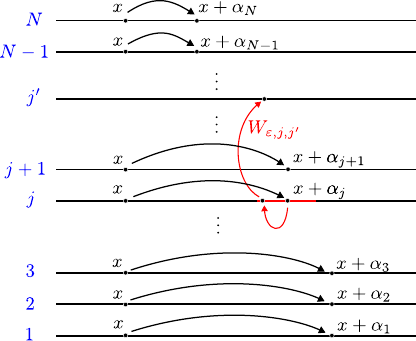}
\caption{\rev{Illustration of noisy rotational dynamics on circular fibres. Each horizontal black line represents a copy of $S^1$. In one iteration of the dynamics, a point $x$ on the $j^{\rm th}$ fibre is (i) rotated by $\alpha_j$ to $x+\alpha_j$, (ii) a random rotation drawn uniformly from the interval $(-\delta,\delta)$ is added (shown as a red interval), and (iii) finally one moves from fibre $j$ to fibre $j'$ according to the Markov chain with conditional probability $W_{\e,j,j'}$. 
In this figure, fibres 1, 2, and 3 all rotate with the same speed, $\alpha_1=\alpha_2=\alpha_3$. Similarly, the fibres $j$ and $j+1$ rotate at the same speed, as do the final group of fibres numbered $\ldots,,N-1,N$.}}
    \label{fig:model}
\end{figure}
We are particularly interested in the case where several neighbouring circular fibres rotate with the same speed, \rev{as in  Figure \ref{fig:model33}}.
\begin{figure}
    \centering
\includegraphics[width=0.5\linewidth]{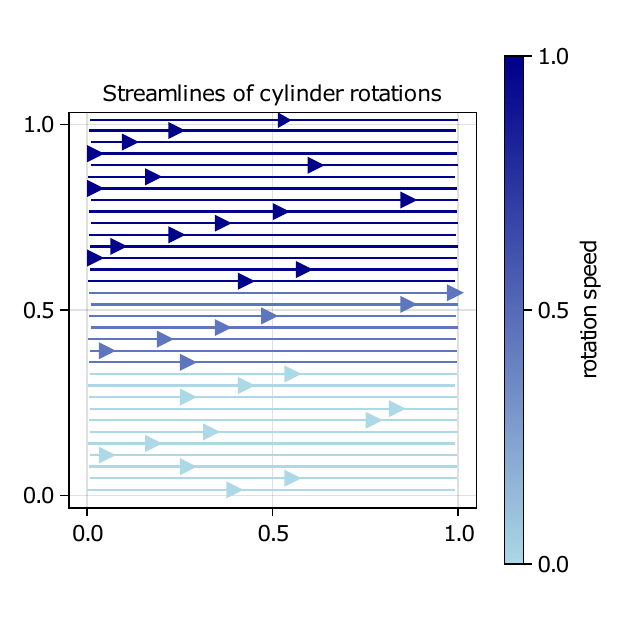}
    \caption{Streamlines of the cylinder rotation model with three bands of circular fibres with widths $L_1=11, L_2=7, L_3=15$ and rotation speeds  $\beta_1=\pi/20, \beta_2=e/7, \beta_3=1/\sqrt2$, respectively.}
    \label{fig:model33}
\end{figure}
The dynamics on these groups of fibres model bands of fluid rotating with a common speed.

The addition of noise is a crucial selection procedure that promotes important eigenfunctions of the transfer/Koopman operator by associating them with large-magnitude eigenvalues.
This heuristic is classical for the eigenvalue 1 where the eigenfunction corresponds to the invariant measure \cite{Kifer1974, Khasminskii,LSY, Alves2007}, \rev{which describes the state-space distribution of long trajectories,} and for large positive real eigenvalues in connection with almost-invariant sets and metastability \cite{DellnitzJunge99}.
Moreover, the addition of noise may also be useful for proving convergence of numerical schemes \cite{DellnitzJunge99,keller1999stability,crimmins2020fourier}.
It has also been applied in the context of thermodynamic formalism, where noise perturbations identify equilibrium states by associating them with large-magnitude eigenvalues of suitable operators \cite{BCL2024,BC2025}. 
Further motivations for considering the addition of noise are:
\begin{enumerate}
    \item models and observations are imperfect and often subject to small errors; 
    \item \rev{many} numerical approaches to estimating transfer and Koopman operators introduce artificial noise or diffusion through reduction to a finite-rank representation of the operator; 
    \item the addition of noise enables formal discussion of spectra and eigenfunctions of transfer operators acting on simple spaces such as $L^2$;
    \item zero-noise limits are difficult limits of singular type, yet  remain relevant and are the type observed in numerical experiments.
\end{enumerate}


This work builds fundamental theory to support the heuristic notion that the addition of noise \textit{promotes eigenfunctions connected with dominant cycles toward the largest-magnitude complex eigenvalues}. 
We achieve this by conducting a formal analysis of a model system that possesses the relevant cyclic properties. 
We perturb this model with noise and carefully study the zero-noise limit, fully characterising the limiting eigenvalues and the form of the limiting eigenfunctions. 
Moreover, we prove quadratic response of the eigenvalues and linear response of the eigenfunctions, the form of which demonstrates a \textit{strong robustness of the \rev{support-localisation of the} eigenfunctions to a range of \rev{small} noise}.

To introduce noise to our model we use a \rev{discrete-time Markov process} between the circular fibres \rev{governed by an $N\times N$ stochastic transition matrix}
$W_\e\ge 0$ that converges to the identity matrix as $\e\to 0$.
\rev{For example, if we have $N=33$ fibres as in Figure \ref{fig:model33}, the matrix \begin{equation}
\label{Weps1d}
    W_\e :=\begin{pmatrix}
1-\frac{\e}{2} & \frac{\e}{2} & & &  &\\
\frac{\e}{2} & 1-\e & \frac{\e}{2} &  &  & \\
 & \frac{\e}{2} & 1-\e & \frac{\e}{2} &  &  \\
 &  &  \ddots& \ddots & \ddots & \\
 &  &  & \frac{\e}{2} & 1-\e & \frac{\e}{2} \\
 &  &  &  & \frac{\e}{2} & 1-\frac{\e}{2} 
\end{pmatrix}.
\end{equation}
will be suitable for our model.
}
Additionally, on each circular fibre we apply i.i.d.\ additive noise.
In summary, the dynamics $T:M\to M$ is given by 
\rev{$T(j,x) = (j, x+\alpha_j) + (\gamma, \eta)$, where $\gamma$ is selected according to the Markov chain $W_{\epsilon}$ conditioned on being in state $j$, 
and $\eta$  is} an i.i.d.\ random variable with values uniformly distributed in $[-\delta,\delta]$.
\rev{Equivalently,}
\rev{for each $\e>0$ we consider a Markov chain $X_n^{\e,\delta}=(I_n^\e,Y_n^{\e,\delta})$ on $M$, where $I_n^\e$ is a 
Markov chain on $\{1,\ldots,N\}$ with transition matrix $W_\e$ satisfying
$
W_{\e,i,j}=\mathbb P\left(I_1^\e=j\,\middle|\,I_0^\e=i\right)$.
The second component evolves according to
$
Y_{n+1}^{\e,\delta}:=Y_n^{\e,\delta}+\alpha_{I_n^\e}+\eta_n^{\delta},
$
where $(\eta_n^{\delta})_{n\ge0}$ are i.i.d.\ random variables, each uniformly distributed on $[-\delta,\delta]$ as above.}

For this model, we demonstrate  a separable structure of the eigenfunctions of the transfer operator (Proposition \ref{prop:1})  and use this structure to generate a series of precise results regarding zero-noise limits.
We prove that as the noise $\e$ approaches zero, the eigenvalues converge to a complex exponential given by one of the rotation speeds, \rev{linearly in $\e$ with an explicit bound for the distance from the limit} (Proposition \ref{prop:1}).
When there are bands of neighbouring \rev{circular} fibres that share a rotation speed,  \textit{clusters} of eigenvalues approach a single complex exponential (Theorem \ref{thm:1} and Figure \ref{fig:pertmatrixfig}(left)). 
We prove that in the $\e\to 0$ limit, the corresponding eigenfunctions approach an orthonormal basis supported on the band (Theorem \ref{thm:1} \rev{and Figure \ref{fig:fullefuns}}).
Despite having multiple eigenvalues converging to a single eigenvalue, we develop a quadratic response formula for the eigenvalues and linear response formulae for the eigenfunctions (Theorem \ref{thm:2}).
The supports of the latter \rev{are also restricted to their associated band}, resulting in this support localisation property being particularly robust for small noise levels.
Proposition \ref{prop:LRA} covers the response of \rev{eigenvalues and eigenvectors of the transfer operator} with respect to \rev{perturbations of} the rotation speeds $\{\alpha_j\}_{j=1}^N$.

An outline of the paper is as follows.
\rev{In Section \ref{sec:2},} we construct our model, \rev{with dynamics as shown in Figure \ref{fig:model}}, and state our main results.
\rev{Section \ref{sec:2} contains a detailed running example that illustrates and interprets the conclusions of our theory.
}
\rev{For this running example, Figure \ref{fig:model33} illustrates the fluid-band structure of the model,} Figure \ref{fig:pertmatrixfig} summarises the behaviour of the most important (lowest-order) eigenvalues and eigenvectors for modest positive $\e$, Figure \ref{fig:fullefuns} visualises the magnitude and argument of the corresponding eigenfunctions on the full cylindrical domain, and Figure \ref{fig:respmatrixfig} summarises the response behaviour of these eigenvalues and eigenvectors.
The proofs of the main results are carried out in Sections \ref{sec:3}--\ref{sec:alpharesp}.

\section{Model and statements of main results}
\label{sec:2}
\subsection{The model}




 The Perron--Frobenius operator corresponding to the Markov chain on the fibres $\{1,...,N\}$ is given by
 \begin{equation}
     \label{Pbase}
\mathcal{P}_{\mathrm{base},\e}f(j'):=\sum_{j=1}^N W_{\e,j,j'} f(j),
\end{equation}
On the $j^{th}$ band, the annealed transfer operator 
$\mathcal{P}_{\delta\restriction j}:L^1(\mathbb{S}^1)\to L^1(\mathbb{S}^1)$ is given by 
\begin{equation}
\label{jfibreTO}
\mathcal{P}_{\delta\restriction j} h(x):=\frac{1}{2\delta}\int_{[-\delta,\delta]} h(x-\alpha_j-\eta)\ d\eta.
\end{equation}

The product of the discrete uniform probability measure on $\{1,\ldots,N\}$ and normalised Lebesgue measure on $\mathbb{S}^1$ is denoted by $m$.
Using the form of $T$, and letting $F\in L^1(M)$, the annealed transfer operator with respect $m$ is given by 
\rev{$$\mathcal{P}_{\e,\delta}F(j',x)=\sum_{j=1}^N W_{\e,j,j'}\frac{1}{2\delta}\int_{[-\delta,\delta]} F(j , x-\alpha_{j}-\eta)\ d\eta.$$}
We now write $F(j,x)=f(j)h_j(x)$ to highlight the skew-product nature of $T$ and $\mathcal{P}_{\e,\delta}$.
\rev{\begin{align}
\nonumber\mathcal{P}_{\e,\delta}F(j',x)&=
\sum_{j=1}^N W_{\e,j,j'}f(j)\frac{1}{2\delta}\int_{[-\delta,\delta]} h_{j}(x-\alpha_{j}-\eta)\ d\eta \\
&=\sum_{j=1}^N W_{\e,j,j'}f(j)\mathcal{P}_{\delta\restriction j}h_{j}(x) 
=\mathcal{P}_{\mathrm{base},\e}f(j)\cdot\mathcal{P}_{\delta\restriction j}h_{j}(x). 
\label{skewproduct}
\end{align}}

Because each $\mathcal{P}_{\delta\restriction j}$ is the transfer operator for a rotation on the $j^{th}$ circle by $\alpha_j$ (plus uniform bounded noise), the eigenfunctions of $\mathcal{P}_{\delta\restriction j}$ are complex exponentials of increasing order, \emph{independent of $j$}.
Thus, as we shall shortly see, the eigenfunctions of $\mathcal{P}_{\e,\delta}$ are \textit{separable in the variables $j$ and $x$}.


\subsection{General behaviour of eigenvalues and eigenfunctions}

\subsubsection{The case of $\e=0$.}

Before we begin our main analysis, it is informative to consider the simple situation where $\e=0$ and $\delta\ge 0$.
\rev{This corresponds to no noise between the circular fibres, and $\delta$-noise along fibres.}
With no noise connecting the distinct copies of $\mathbb{S}^1$, the rotations on each circle all individually contribute to generate the full eigenspectrum of $\mathcal{P}_{0,\delta}$.

\begin{proposition}
\label{prop:eps0}\rev{
The spectrum of $\mathcal{P}_{0,\delta}$  consists of the $N$-fold union of the eigenvalues $\bigcup_{j=1}^N \sigma(\mathcal{P}_{\delta\restriction j})$, where}
\begin{align}
\rev{\sigma(\mathcal{P}_{\delta\restriction j})=\bigcup_{k\in\mathbb{Z}}\left\{\frac{\sin(2\pi k\delta)}{2\pi k\delta}e^{-2\pi i k\alpha_j}\right\}\cup\{0\}.} \label{eq:eq1}
\end{align}
\end{proposition}
\begin{proof}
\rev{
In this case $W_\e = \mathrm{Id}$, therefore \eqref{skewproduct} implies that given $F(j,x) = \sum_{k\in\mathbb Z} f_k(j) e^{2\pi i k x}\in L^2(M)$.
Using \eqref{jfibreTO} one has 
\begin{align*}
    \mathcal P_{0,\delta} F(j,x) &= \sum_{k\in\mathbb Z}  f(j) \mathcal P_{\delta\restriction j} (e^{2\pi i k x})= \sum_{k\in\mathbb Z} f(j)\frac{1}{2\delta} \int_{-\delta}^{\delta} e^{2\pi i k (x- \alpha_j - \eta)} \d \eta\\
     &=\sum_{k\in\mathbb Z} \left(\frac{\sin(2\pi k\delta )}{2\pi k \delta}  e^{-2\pi i k \alpha_j } \right) f(j)  e^{2\pi i k x}, 
\end{align*}
where $\sin(0)/0$ should be interpreted as $1$ in the above equation. Thus, by comparing the Fourier modes of $\mathcal P_{0,\delta} F$ and $\lambda F$ we obtain \eqref{eq:eq1}.}
\end{proof}
\rev{The modes $k=\pm 1$ are the primary modes of oscillation corresponding to a single oscillation around the circle, as shown by the $e^{2\pi ikx}$ term in the final display equation above. In practice, selecting $k=\pm 1$ makes the extraction of the $\alpha_j$ from an eigenvalue straightforward because of the argument term of $e^{-2\pi ik\alpha_j}$.}
\rev{For positive $\delta$ the damping factor $\frac{\sin(2\pi k\delta)}{2\pi k\delta}$ suppresses higher-order Fourier modes. 
In particular, the  $k=\pm1$ modes have the largest magnitude (apart from the trivial constant mode),}
\rev{making the identification of the $k=\pm 1$ modes simpler.}

\begin{remark}
All of the results in Section \ref{sec:2} can be easily generalized from the state space $M=\{1,\ldots,N\}\times S^1$ to a state space $M = \{1, \ldots, N\} \times \mathbb{T}^d$ where:
\begin{enumerate}
    \item the $\alpha_j$-rotations on the circle $S^1 = \mathbb T^1$ are replaced with $(\alpha_{j,1},\ldots, \alpha_{j,d})$-translations on the $d$-torus $\mathbb T^d$. 
\item the discrete fibres $\{1,\ldots,N\}$ are no longer discretely filling an interval, but discretely fill a connected subset of $\mathbb{R}^{d'}$.
 \end{enumerate}
Such a situation allows one to create models of rather general rotational dynamics.
For example, a simple extension to $d=1$ and $d'=2$ would model fluid rotation around a solid three-dimensional torus, with the $N$ fibres discretely filling out the non-periodic two-dimensional cross-section of the solid torus.
Subsets \rev{$B_s$} of $\{1,\ldots,N\}$ called \emph{band \rev{members}} (see Definition \ref{def:cluster1}) represent solid ``tubes'' \rev{$B_s\times S^1\subset \mathbb{R}^2\times S^1$} of three-dimensional fluid that rotate at a common angular speed inside a solid three-dimensional torus.
Alternatively -- remaining in a solid three-dimensional torus -- we may set $d=2$ and $d'=1$ to model translations on two-dimensional nested toral shells within the solid torus.
Bands (subsets of $\{1,\ldots,N\}$) represent nested families of solid tori \rev{$B_s\times \mathbb{T}^2\subset \mathbb{R}\times\mathbb{T}^2$} (like Matryoshka dolls) with common translations on each member of the family. 

In this general setting, the spectrum of $\mathcal{P}_{0,\delta}$ is 
\begin{equation*}
\bigcup_{j=1}^{N}\sigma(\mathcal{P}_{\delta\restriction j})=\bigcup_{j=1}^{N}\bigcup_{k\in\mathbb{Z}^d}\left\{\prod_{q=1}^d\frac{\sin(2\pi k_q\delta)}{2\pi k_q\delta}e^{-2\pi i k_q\alpha_{j,q}}\right\}\cup\{0\}
\end{equation*}
\rev{where $\alpha_{j,q}$ is the rotation speed on fibre $j$ along coordinate direction $q$.}
Concerning the extension from $S^1$ to $\mathbb{T}^d$, the same proofs of all of our results follow with minor adaptations.
For the extension of the fibres $\{1,\ldots, N\}$ filling out sets in $\mathbb{R}^{d'}$, one only needs to appropriately adjust the matrices $W_\epsilon$.
For example, in $d'$ dimensions, a natural choice would be to construct $W_\epsilon$ from a finite-difference stencil for the Laplace operator on a regular grid in $d'$ dimensions (see \eqref{Weps1d} for such a choice for $d'=1$).
For the sake of simplicity of exposition, we restrict the statements and proofs of all results to the case $d = 1$. 
\end{remark}



\subsubsection{The case of $\e>0$.}
In the remainder of this work, we will assume that $\delta=0$ \rev{(no along-fibre noise)} because
\begin{enumerate}
    \item the effect of $\delta$ on the spectrum is simply to multiply each eigenvalue by $\frac{\sin(2\pi k\delta)}{2\pi k\delta}$, 
    \item the eigenfunctions are unaffected by changes in $\delta$.
    \end{enumerate}
Hereafter we will denote $\mathcal{P}_{\varepsilon,\delta}$ by simply $\mathcal{P}_\varepsilon$.
When $\e>0$ we no longer have a closed form expression for the eigenvalues and eigenfunctions of  $\mathcal{P}_{\e}$, but we can 
characterise them as eigenvalues and eigenvectors of the simple matrix product $D_{k,\alpha} W_\e$ where
\begin{equation}
    \label{Dkalpha}
D_{k,\alpha}=\mathrm{Diag}(e^{-2 \pi i k \alpha_1 }, \ldots, e^{-2 \pi i k \alpha_N }),
\end{equation}
and analyse their behaviour for small $\e>0$ as $\e\to 0$.

\begin{proposition}\label{prop:1}
Let $\alpha =\left(\alpha_1, \ldots, \alpha_N \right)\in \mathbb{R}^N$. 
\begin{enumerate}
    \item For every $\e>0$, \rev{each $k\in \mathbb{Z}$ gives rise to $N$ (generalised) eigenfunctions $F_{k,\e}^{(\ell)}$ of the operator $\mathcal{P}_{\e} : L^2(M) \to L^2(M)$, of the separable form 
    \begin{equation}
    \label{eq:F}   F_{k,\e}^{(\ell)}(j, x) = f_{k,\e}^{(\ell)}(j) e^{2\pi i k x}\qquad\mbox{for all $\ell \in \{1,\dots,N\}$,}
    \end{equation}
    where the corresponding $N$ eigenvalues are denoted $\lambda_{k,\e}^{(\ell)}, \ell=1,\ldots,N$.} 
         
    \item \rev{For sufficiently small $\e>0$ one has 
    \begin{equation}
        \label{gbound}
    |\lambda_{k, \e}^{(\ell)} -e^{-2 \pi i k \alpha_\ell}|\le 2 \max_{1\leq j\leq N} (1- W_{\e,j,j}) \qquad\mbox{for all $k,\ell \in \{1,\dots,N\}$},
    \end{equation}
    and 
    \begin{equation}
     \label{argbound}
\left|\mathrm{arg}\left(\lambda_{k,\e}^{(\ell)}\right) - 2\pi k\alpha_\ell\ (\mathrm{mod}\ 2\pi) \right|\leq 2\pi  \max_{1\leq j\leq N} (1- W_{\e,j,j})\qquad\mbox{for all $k,\ell \in \{1,\dots,N\}$}
\end{equation}}
\end{enumerate}
\end{proposition}
\begin{proof}
    See Section \ref{sec:3.1}.
\end{proof}
\rev{We recall that for practical purposes we are interested in the primary oscillations given by the $k=\pm 1$ modes, as discussed immediately after Proposition \ref{prop:eps0}.}
\rev{The bound in \eqref{gbound} is tight;  see Figure \ref{fig:pertmatrixfig}(left), where the elements of $\{\alpha_j\}_{j=1}^{33}$ take on one of three distinct values. The error bounds are shown as dark blue circles.}
\begin{figure}
    \centering
\includegraphics[width=0.49\linewidth]{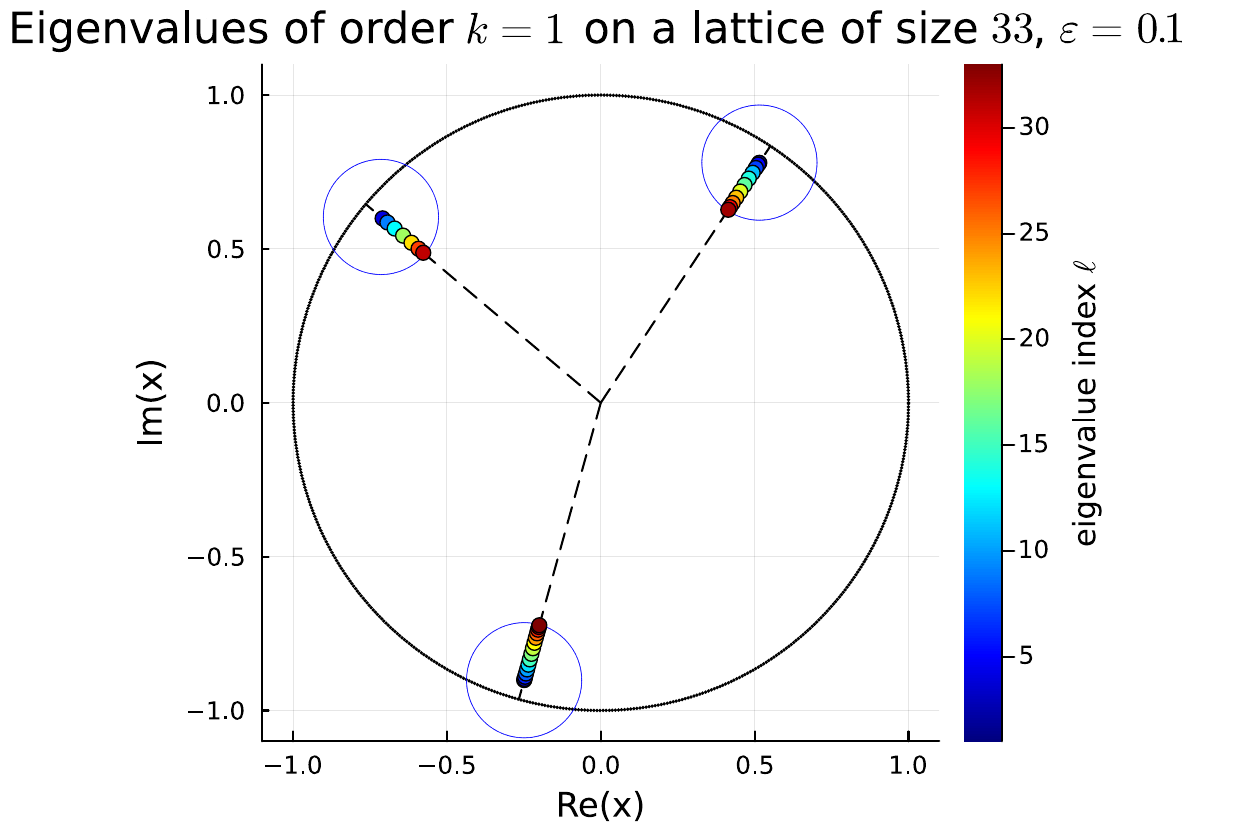}
\includegraphics[width=0.49\linewidth]{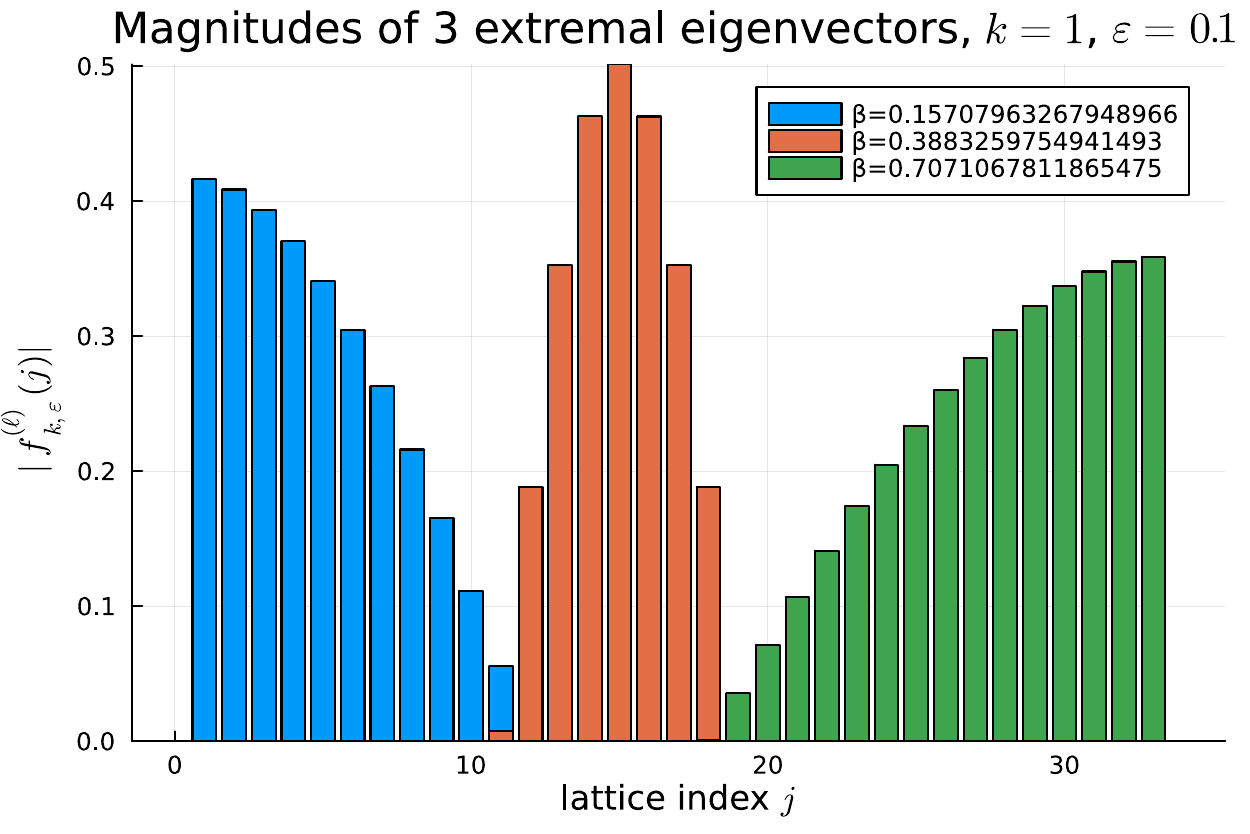}
    \caption{\textit{Left:} The small coloured disks indicate the spectrum $\lambda_{1,\e}$ (for $\delta=\e=0.1$) of the $33\times 33$ matrix $D_{k,\alpha}W_\e = D_{k,\beta,L}W_\e$,  associated to the model from Figure \ref{fig:model33}, where $\beta=(\pi/20, e/7, 1/\sqrt2)$ and $L=(11,7,15)$. The large black circle is the unit circle in the complex plane. Note that the 33 eigenvalues are distinct and that they appear in groups, with $L_i$ eigenvalues nearby $\exp(-2\pi i k \beta_i)$ for $i=1,2,3$. The relevant bounds \eqref{gbound} from Proposition \ref{prop:1} are shown as dark blue circles. 
    \textit{Right:} Plot of $|f_{k,\e}^{(\ell)}(j)|$ vs lattice index $j$ for $\ell=1,2,4$, where the ordering of the eigenvalues are shown in the left panel. We plot the eigenvectors of the largest-magnitude eigenvalue from each of the three groups \rev{corresponding to $k=1$};  the eigenvectors have been normalised to have unit norm.
    }
    \label{fig:pertmatrixfig}
\end{figure}
\begin{figure}
    \centering
\includegraphics[width=0.32\linewidth]{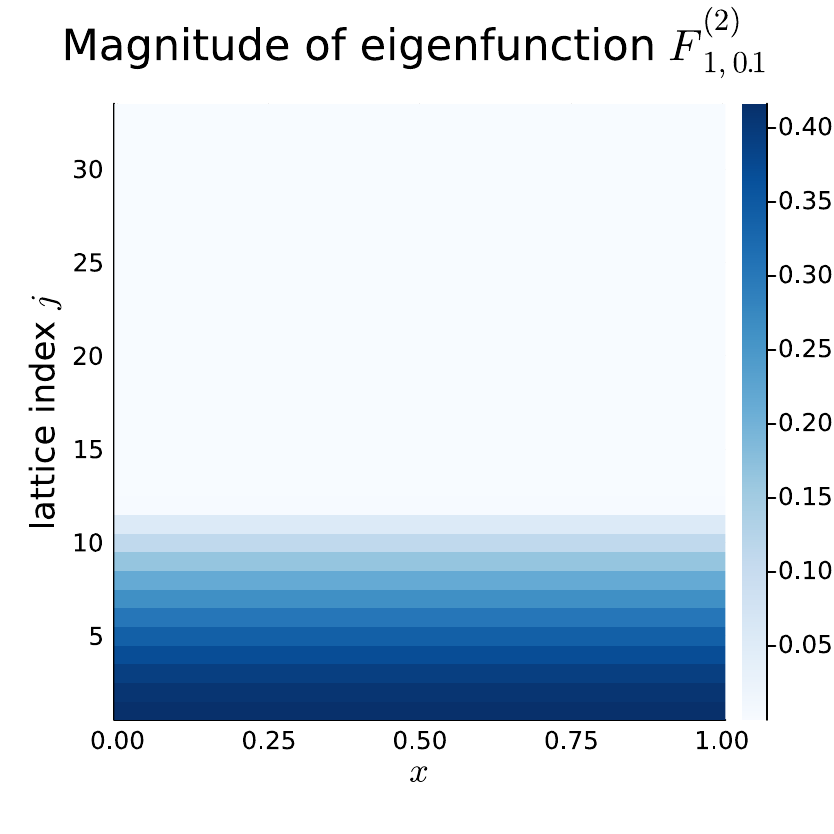}
\includegraphics[width=0.32\linewidth]{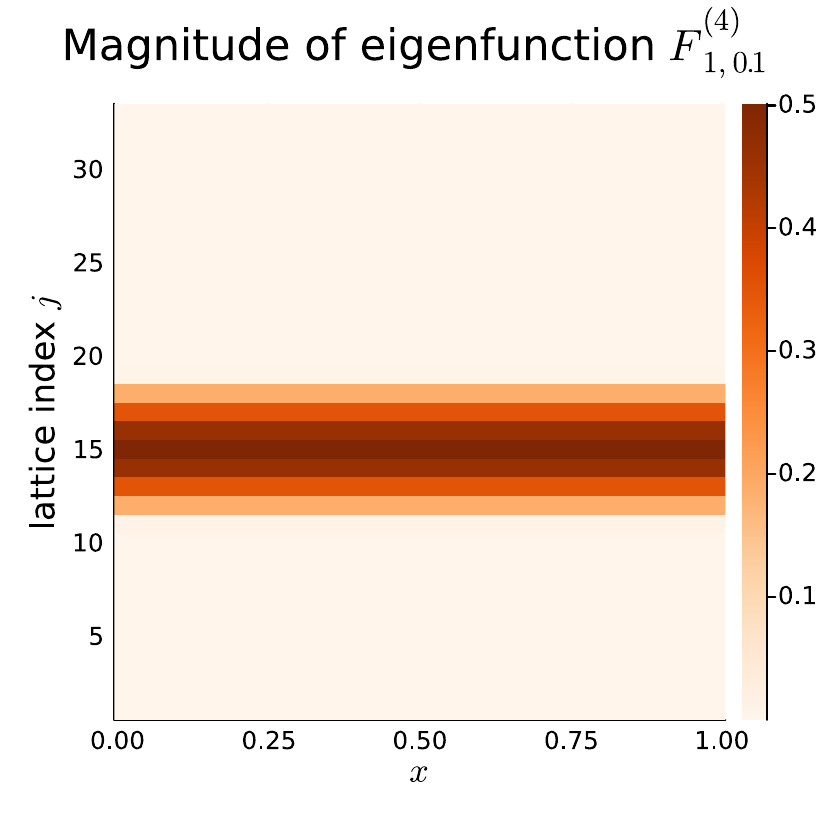}
\includegraphics[width=0.32\linewidth]{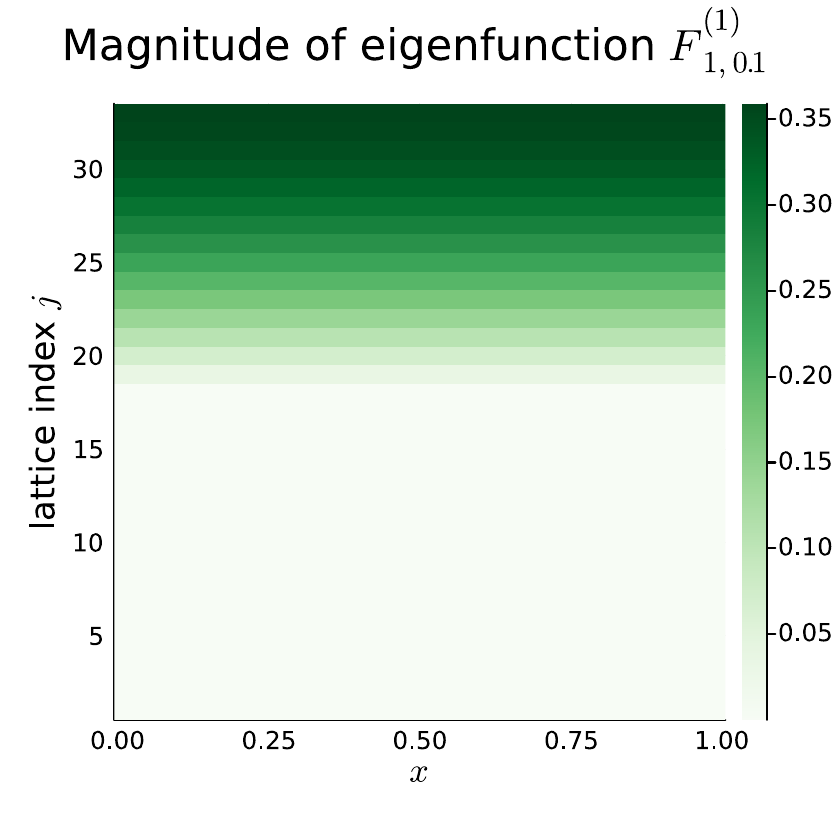}
\includegraphics[width=0.32\linewidth]{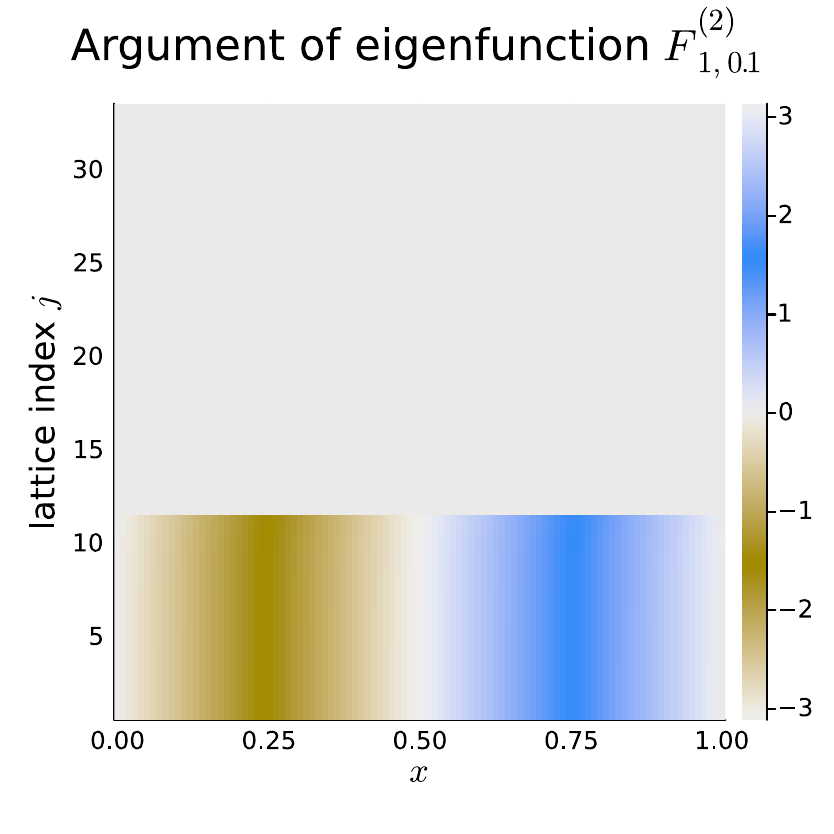}
\includegraphics[width=0.32\linewidth]{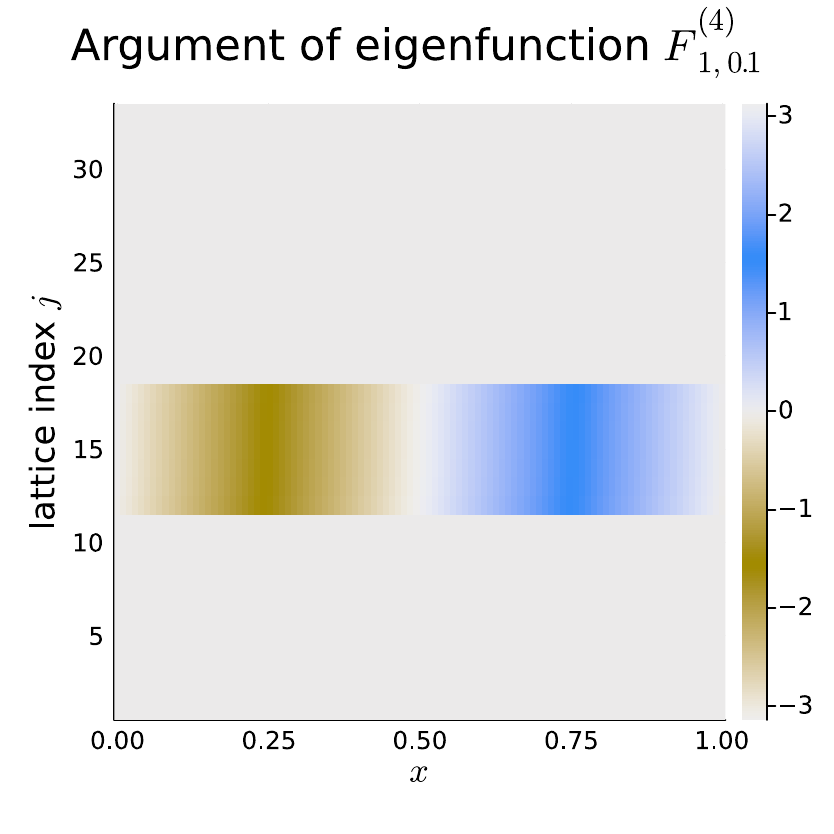}
\includegraphics[width=0.32\linewidth]{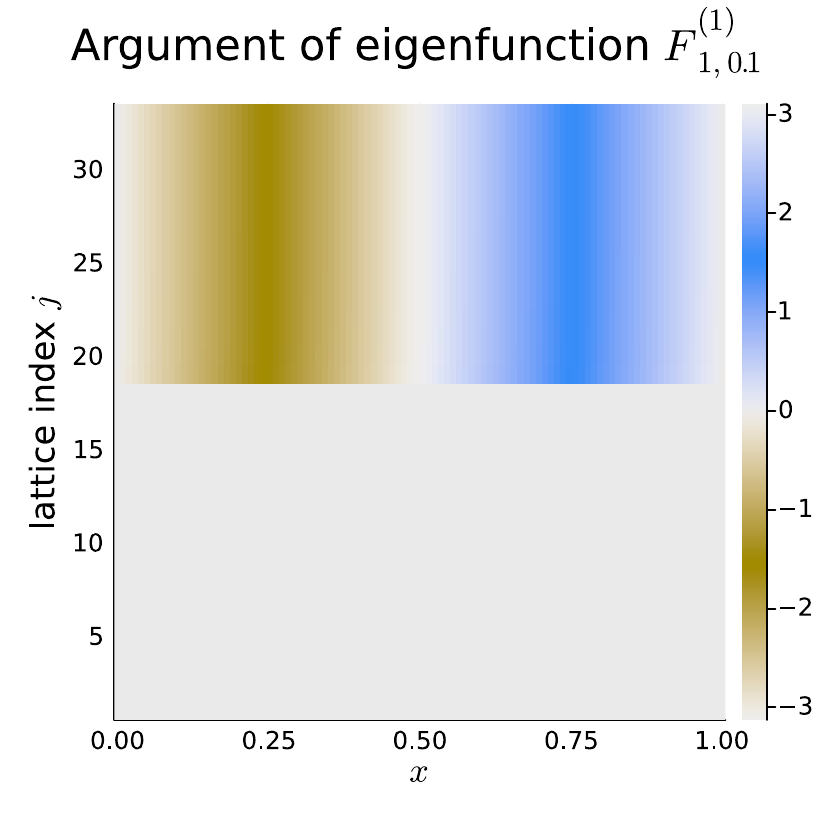}
    \caption{Details of leading eigenfunctions of $\mathcal{P}_\e$ on $M$, constructed from the eigenvectors in Figure \ref{fig:pertmatrixfig}. \textit{Upper row:} Magnitudes of the extremal three eigenfunctions $F_{k,\e}^{(\ell)}(j, x) = f_{k,\e}^{(\ell)}(j) e^{2\pi i k x}$, $\ell=1,2,4$ and $k=1$ on $M=\{1,\ldots,33\}\times \mathbb{S}^1$ corresponding to the eigenvectors $f$ in Figure \ref{fig:pertmatrixfig}, using colourschemes to match the colouring in Figure \ref{fig:pertmatrixfig}. 
    Note that the support of each eigenfunction is approximately restricted to the bands associated with one of the three rotation speeds illustrated in Figure \ref{fig:model}.
    \textit{Lower row:} As for the upper row, but displaying the arguments of eigenfunctions instead of the magnitudes. 
     Note that exactly one complex cycle occurs because $k=1$.
    }
    \label{fig:fullefuns}
\end{figure}

Using Proposition \ref{prop:1}, we assume from now on that the eigenvalues  
$\lambda_{k,\e}^{(1)},\ldots, \lambda_{k,\e}^{(N)}$ are labelled such that  
\begin{align}
    \lambda_{k,\e}^{(\ell)} &\xrightarrow[]{\e \to 0} e^{-2\pi i k\alpha_{\ell}} \ \text{for every}\ \ell \in \{1,\ldots, N\}.\label{eq:lambdanotation}.
\end{align}


\subsection{The case of multiple fibres with common rotation speed}
A particularly interesting situation arises when some $\alpha_j$ are identical for neighbouring $j$s.
This corresponds to a band of fluid rotating at the same speed, and handling this physically relevant situation requires a more careful analysis.
This motivates \rev{the following definition}.


\begin{definition}[$S$-banded]  \label{def:cluster1}
A $N$-vector $\alpha =\left(\alpha_1, \ldots, \alpha_N \right)\in \mathbb{R}^N$ is said to be \emph{$S$-banded} if there exist 
\begin{itemize}
    \item[$(i)$] a natural number $1 \leq S \leq N$,
    \item[$(ii)$] distinct real numbers $\beta_1, \ldots, \beta_S \in \mathbb{R}$, and
    \item[$(iii)$]  a vector of positive integer numbers $L = (L_1,\ldots, L_S)$ satisfying $\sum_{i=1}^N L_i  = N,$ 
\end{itemize}
such that 
$$
\alpha_j = \beta_s \quad \text{if } N_{s-1} < j \leq N_s \quad \text{for some } s \in \{1, \ldots, S\},
$$  
where $N_0<N_1<\ldots < N_{S}=N$ is a sequence of recursively defined natural numbers
$$N_0 =0\ \text{and }N_i = L_i + N_{i-1}\ \text{if }i\in\{1,\dots,S\}$$
In the above context, the vector $\beta = (\beta_1, \ldots, \beta_S)$ is referred to \rev{as a} \emph{band vector}, and  $L$ is called the \emph{band-width vector}. Moreover, given $s\in\{1,\ldots,S\}$ we denote the \it{band members of $\beta_s$} as $B_s=\{j\in \N: N_{s-1}< j\leq N_{s}\}.$ Observe that from the convention adopted in \eqref{eq:lambdanotation} we have that 
$\{\ell\in\{1,\ldots,N\}:\lambda^{(\ell)}_{k,\varepsilon}\to e^{-2\pi i k \beta_s}\mbox{ as $\varepsilon\to 0$}\}= B_s.$
\end{definition}
\begin{example}[$S$-banded]
\,
\rev{\begin{enumerate}
\item any $N$-vector $\alpha=(\alpha_1,\ldots,\alpha_n)$ with distinct entries is $N$-banded,
    \item the $5$-vector $\alpha = (\sqrt{2},\sqrt{2},\sqrt{2},\pi,\pi)$ is $2$-banded with band vector $\beta = (\sqrt{2},\pi)$ and band-width vector $L=(3,2)$,
    \item the $5$-vector $\alpha=(\sqrt{2},\sqrt{2},\pi,\sqrt{2},\sqrt{2})$ is not an $S$-banded vector \rev{for any $1\le S\le N$.}
\end{enumerate}}
\rev{As part of our running example, associated to the 33-fibre case study with $W_\e$ as in \eqref{Weps1d}, we choose $S=3$ bands rotating at incommensurate speeds given by the vector $\beta=(\pi/20, e/7, 1/\sqrt{2})$, and with corresponding band-width vector $L=(11,7,15)$;  see Figure \ref{fig:model33}.}
\end{example}

\rev{Let $\alpha$ be an $S$-banded $N$-vector for some $1\le S<N$, then there is at least one $1\le s\le S$ such that $\beta_s$ has band width $L_s\ge 2$.}
Proposition \ref{prop:1}(2) tells us that for a given Fourier index $k$, $\lambda^{(\ell)}_{k,\e}$ converges to $\exp(-2\pi i k \beta_s)$ for $L_s$-fold $\ell$, this phenomenon is illustrated in Figure \ref{fig:pertmatrixfig} (left).
Therefore, one expects an $L_s$-dimensional subspace spanned by the corresponding eigenvectors $F^{(\ell)}_{k,\e}$ to converge to some $L_s$-dimensional space, but the question of convergence of individual eigenvectors and their characterisation is still unclear.
\rev{To study this limit we must slightly constrain the one-parameter family of matrices $W_\e$ that control noise between circular fibres.}

\begin{definition}[$L$-admissible family of matrices] Given an $S$-vector $L=(L_1,\ldots,L_s)$,
\label{Ladmiss}
 the family of $N\times N$ matrices $\{W_\e = \mathrm{Id} + \e \dot{W}\}_{\e\geq 0}$  is \rev{called} $L$-admissible if the symmetric matrix $\dot{W}$ satisfies:
\begin{enumerate}
    \item \( \dot{W}_{ij} \geq 0 \) for \( i \neq j\in\{1,\ldots, N\}\) and $\sum_{j=1}^{N} \dot{W}_{ij} = 0, \quad \forall i \in \{1, \dots, N\}$.
    \item  $\dot{W}$ has $N$ distinct eigenvalues.
    \item For each $s\in\{1,\ldots,S\}$, the $L_s\times L_s$ submatrix $\hat{W}_s := \big[\dot{W}_{jk} \big]_{N_{s-1} < j,k \leq N_s}$ has \( L_s \) distinct eigenvalues.
\end{enumerate}  
\end{definition}

\begin{example}[Noise between circular fibres]
\rev{The matrix $W_\e$ in \eqref{Weps1d} is constructed as $W_\e:=\mathrm{Id}+\e \dot W$, where 
\begin{equation}
    \label{Wdot1d}
\dot{W}:=\begin{pmatrix}
    -\frac{1}{2} & \frac{1}{2} & & &  &\\
\frac{1}{2} & -1 & \frac{1}{2} &  &  & \\
 & \frac{1}{2} & -1 & \frac{1}{2} &  &  \\
 &  &  \ddots& \ddots & \ddots & \\
 &  &  & \frac{1}{2} & -1 & \frac{1}{2} \\
 &  &  &  & \frac{1}{2} & -\frac{1}{2} 
\end{pmatrix}.
\end{equation}
This matrix clearly satisfies items (1) and (2) of Definition \ref{Ladmiss} for  any band-width vector $L$.
Item (3) of Definition \ref{Ladmiss} holds because submatrices of $\dot W$ correspond to central difference approximations of the Laplace operator on an interval with either mixed Neumann/Dirichlet boundary conditions (for the first and last submatrix) or Dirichlet boundary conditions (for ``internal'' submatrices) (see e.g.\ \cite[Section 3]{Eigen}).
Thus, $W_\e$ is an $L$-admissible family for any bandwidth vector $L$.
}

\end{example}


In the $\e\to 0$ limit,  instead of considering eigenfunctions, we use eigenprojections, as eigenfunctions can always be multiplied by unit-norm complex scalars, which alters their phase. 
This phase freedom \rev{makes the proofs of convergence and linear response more difficult.}
Eigenprojections, on the other hand, remain well defined and invariant under such transformations.
Theorem \ref{thm:1} establishes that if the elements of the band vector are incommensurate, the eigenprojections of \( \mathcal{P}_\e \) converge as $\varepsilon\to 0$. 
Furthermore, the limiting projections can be precisely characterized in terms of eigenvectors of specific matrices. 

\begin{theorem}\label{thm:1}
Let $N \in \mathbb{N}$, $1\leq S\leq N$, $L = (L_1, \ldots, L_S)$ be a band-width vector and $\{W_\e = \mathrm{Id}+\e \dot{W}\}_{\e\geq 0}$ an $L$-admissible family of matrices. 
Then for each $$\beta \in \Gamma:=\{(\beta_1,\ldots,\beta_S)\in \R^S: e^{-2\pi i k \beta_{s_1}}\neq e^{-2\pi i k \beta_{s_2}}\ \forall k\in\mathbb Z, \forall s_1,s_2\in\{1,\ldots,S\}\}$$ and $k\in\mathbb Z$:
\begin{enumerate}
\item For sufficiently small $\e$ (depending on $k)$, the $N$ generalised eigenfunctions $F_{k,\e}^{(\ell)}$ from \eqref{eq:F} become $N$ eigenfunctions with distinct eigenvalues $\lambda_{k,\e}^{(\ell)}$, $\ell=1,\ldots,N$.
\item 
For \rev{sufficiently small $\e>0$ (depending on $k$)} and each $\ell=1,\ldots,N$,
recall $$F_{k}^{(\ell)}(j,x) = f_k^{(\ell)}(j)e^{2\pi k i x}$$ and let $\Pi_{k}^{(\ell)},\Pi_{k,\e}^{(\ell)} :L^2(M)\to L^2(M)$ denote $L^2$-orthogonal projections onto the spans of the eigenfunctions $F^{(\ell)}_{k,\e}$ and $F^{(\ell)}_{k}$, respectively.
In the $\e\to 0$ limit there exists a limiting orthonormal basis $\{f_k^{(\ell)}\}_{\ell=1}^N$ of $\mathbb C^N$ such that 
$$\Pi_{k,\e}^{(\ell)} \xrightarrow[]{\e\to 0} \Pi_{k}^{\ell}\ \text{in the }L^2(M)\text{-operator norm}. $$

We may further characterise the $\{f_{k}^{(\ell)}\}_{\ell=1}^N$ as follows: 
\begin{itemize}
\item[$(a)$] If $S=1$ \rev{(one band, all fibres rotate with the same speed)} or $k=0$, $\{f_k^{(\ell)}\}_{\ell=1}^N$ is an orthonormal eigenbasis of the $N\times N$  matrix {$\dot{W}$}.

\item[$(b)$] If $S>1$ \rev{(multiple bands)} and $k\in \mathbb Z\setminus \{0\}$,
  $\{f_k^{(\ell)}\}_{\ell=1}^N$ is an orthonormal eigenbasis of the matrix \begin{equation}
      \label{PkbL}
\hat{P}_{k,\beta,L}:= D_{k,\beta,L} \hat{W}_{L}
  \end{equation}
where 
\begin{align*}
    D_{k,\beta,L} := \begin{pmatrix}
    e^{-2\pi i k \beta_1} \mathrm{Id}_{L_1} &  &  \\
    & \ddots & \\
    &  & e^{-2\pi i k \beta_S} \mathrm{Id}_{L_S},
\end{pmatrix},
\end{align*}  
\begin{align}
    \hat{W}_{L} := \begin{pmatrix}
     \hat{W}_1 &  &  &  &  \\
    & \hat{W}_{2} &  &  &  \\
    &  & \ddots &  &  \\
    &  &  & \hat{W}_{S-1} &  \\
    &  &  &  & \hat{W}_{S}
\end{pmatrix},\label{What}
\end{align}
and  
 $\mathrm{Id}_n$ is the $n \times n$ identity matrix.
Moreover, for each $s\in\{1,\ldots,S\}$, if $\ell$ is a band member of $\beta_s$,   the support of \rev{$F_k^{(\ell)}$ is contained in $\{N_{s-1}+1,\ldots,N_s\}\times S^1$}.
\end{itemize}
\end{enumerate}
\end{theorem}
\begin{proof}
    See Sections \ref{sec:3.2}, \ref{sec:3.3} and \ref{sec:4}.
\end{proof}
\rev{In our main case of interest, namely $1<S<N$, Theorem \ref{thm:1} provides explicit formulae for the $\e\to 0$ limiting eigenfunctions and constrains the supports of these limiting eigenfunctions to their associated band fibres.}  
\begin{example}[Small $\e>0$ and $\e\to 0$ behavior of eigenvalues and eigenfunctions]
 \rev{We continue our example where $W_\e$ is given by \eqref{Weps1d}, $\beta=(\pi/20, e/7, 1/\sqrt{2})$, and  $L=(11,7,15)$.}
 
 \rev{\textit{Eigenvalues:} For $\e=\delta=0.1$, Figure \ref{fig:pertmatrixfig}(left) shows the eigenspectrum of $\hat{P}_{1,\beta,L}$ (i.e.\ for $k=1$).
For sufficiently small $\e$, Proposition \ref{prop:1}(2) states that these eigenvalues cluster about the 3 angles in the vector $2\pi\beta$;  the latter is indicated by dotted rays. 
The bound \eqref{gbound} (including a factor $\sin(2\pi\delta)/(2\pi\delta)$) is indicated with dark blue circles and is seen to be very close to optimal. 
Theorem \ref{thm:1} states that for sufficiently small $\e$ we obtain 33 distinct eigenvalues for each $k\in\mathbb{Z}$, with corresponding one-dimensional eigenspaces.}

\rev{\emph{Eigenfunctions:} Proposition \ref{prop:1} states that all eigenfunctions of $\mathcal{P}$ are of the form $F_{k,\e}^{(\ell)}(j, x) = f_{k,\e}^{(\ell)}(j) e^{2\pi i k x}$, $\ell=1,\ldots,33; k\in\mathbb{Z}$.
The magnitudes of $f_{k,\e}^{(\ell)}$ for $\ell=1,2,4$ (those eigenvectors of $D_{k,\beta,L}W_\e$ corresponding to the three largest-magnitude eigenvalues) are shown in Figure \ref{fig:pertmatrixfig}(right).
Importantly, note that because of the particular form of $F_{k,\e}^{(\ell)}(j, x)$ given above, the magnitudes of the full eigenfunctions $F^{(\ell)}_{k,\e}$ on $M$ are constant along the circular fibres and therefore follow precisely the magnitudes given by the eigenvectors $f^{\ell}_{k,\e}$, such as those in Figure \ref{fig:pertmatrixfig}(right);  see Figure \ref{fig:fullefuns}.}

\rev{The upper row of Figure \ref{fig:fullefuns} can be compared with Figure \ref{fig:lorenz}(left).
In both cases, the magnitudes of the complex eigenfunctions are large in the vicinity of a band of the state space rotating at approximately the same speed.
Moreover, the magnitude of the eigenfunction shown in Figure \ref{fig:lorenz}(left) is relatively constant along the approximately rotational trajectories on each wing.
The lower row of Figure \ref{fig:fullefuns} displays the evolution of the argument of the leading eigenfunctions, which passes through a single cycle because we have chosen the lowest order of $k=1$.
Similarly in Figure \ref{fig:lorenz}(right), we see a single cycle corresponding to a single loop around each wing of the Lorenz attractor.
\textit{Thus, the leading complex eigendata provides complete information on the various speeds or periods of rotational motion and the parts of the state space where that motion occurs.}}

\rev{As $\e\to 0$, Theorem \ref{thm:1} states that we obtain a limiting orthonormal basis of eigenfunctions $\{F_k^{(\ell)}\}_{\ell=1,\ldots,33; k\in\mathbb{Z}}$, and that the limiting eigenvectors $f_k^{(\ell)}$ have supports restricted to one of the three bands.
This support restriction is already evident in Figure \ref{fig:pertmatrixfig}(right) and Figure \ref{fig:fullefuns} when $\e=0.1$.
}
\end{example}


\rev{To analyse a general dynamical system, including e.g.\ the Lorenz system in Section \ref{sec:1}, one would:}
\begin{enumerate}
\item \rev{Construct a numerical estimate of the transfer or Koopman operator that is consistent with convolving the true operator with small noise, so that our theoretical findings should apply. There are many such approaches, for example, Ulam's method \cite{Ulam1964Problems,keller1999stability,FroylandJungeKoltai2013}, direct addition of noise or kernel convolutions \cite{DellnitzJunge99, Giannakis2019DataDriven, crimmins2020fourier}, local kernel evaluations \cite{Williams2015KernelKoopman,FGLPS21} etc. 
In the Lorenz example in Figure \ref{fig:lorenz} we used the simple kernel-based Markov chain approach in \cite{FGLPS21}.}
\item \rev{Compute the largest-magnitude eigenvalues of the transfer or Koopman operator estimate, and identify the largest-magnitude complex eigenvalues that correspond to the primary ($k=\pm 1$) complex eigenfunctions for each cycle. These will often correspond to the complex eigenvalues of largest magnitude and their eigenfunctions will by definition have the least oscillatory phases. In the Lorenz example in Figure \ref{fig:lorenz} this was the second-largest-magnitude eigenvalue after the eigenvalue 1.}
\item \rev{The argument of a complex eigenvalue provides an estimate of the corresponding cycle period (Part 2 of Proposition \ref{prop:1}). 
If a single application of the transfer operator corresponds to evolution for $\tau$ units of time and the argument of the primary complex eigenvalue is $2\pi\alpha$, then the period of the cycle is $\tau/\alpha$.}
\item \rev{The argument of the corresponding complex eigenfunction describes motion around the cycle. In particular, isocontours of the eigenfunction argument correspond to the same phase of the cycle (Part 1 of Proposition \ref{prop:1});  see Figure \ref{fig:lorenz}(right) for the Lorenz example.}
\item \rev{The identified cycle is approximately localised in the support of the corresponding eigenfunction (Part 2(b) of Theorem \ref{thm:1}). In Figure \ref{fig:lorenz}(left), the darker blue indicates a stronger match of cycle period.}
\rev{Our linear response results in the following subsection state that this localisation is robust to small changes in the noise level provided the noise described by $\dot W$ is local in phase space (see Remark \ref{rmk:resp}).}


\end{enumerate}



\subsection{Quadratic and linear response}

Proposition \ref{prop:1} tells us that for each Fourier index $k\in\mathbb Z$, 
in the $\e\to 0$ limit the eigenvalues $\lambda_{k,\e}^{(\ell)}$, $\ell=1,\ldots,N$ of $\mathcal{P}_\varepsilon$ each converge to $e^{-2 \pi i k \alpha_\ell}$, while Theorem \ref{thm:1} says that the eigenprojections $\Pi_{k,\e}^{(\ell)}$ associated with the  eigenfunctions $F^{(\ell)}_{k,\varepsilon}$ of $\mathcal{P}_\varepsilon$ converge to $L^2$-orthogonal projections $\Pi_k^{(\ell)}$ onto $\s\{F_k^{(\ell)}\}$ which have explicit eigenfunction characterisations \rev{and support localisations}.
A fundamental question arising from these results is the rate of convergence of $\Pi_{k,\e}^{(\ell)}$ and $\lambda_{k,\e}^{(\ell)}$ as $\e \to 0$.
If all $\alpha_\ell$ are distinct, one may hope that the corresponding one-dimensional eigenprojections  $\Pi^{(\ell)}_{k,\e}$ are regular functions of $\e$ at $\e=0$.
On the other hand, if we have nontrivial bands --  i.e.\ $S<N$ in Definition \ref{def:cluster1} -- then we would expect at best regular dependence of the various $L_s$-dimensional eigenspaces, corresponding to the $L_s$ repeated eigenvalues $e^{-2\pi i k\beta_s}$, at $\e=0$.
Surprisingly, in the nontrivially banded situation, we can in fact demonstrate quadratic response of \textit{all individual eigenvalues} $\e\mapsto \lambda_{k,\e}^{(\ell)}$ at $\e=0$, and linear response of \textit{all individual eigenprojections} $\e\mapsto \Pi^{(\ell)}_{k,\e}$ at $\e=0$.
To further illustrate why this is surprising, we recall that the standard approach to establishing linear response of the eigenprojections is to differentiate the equation  $\mathcal{P}_\e \Pi_{k,\e}^{(\ell)} = \lambda_{k,\e}^{(\ell)} \Pi_{k,\e}^{(\ell)}$  
with respect to $\e$, where we use the notation $\dot{}$ to denote $\frac{\partial}{\partial \e}$. This differentiation yields the identity  
$$\mathcal{P}_\e \dot{\Pi}_{k,\e}^{(\ell)} + \dot{\mathcal{P}}_\e \Pi_{k,\e}^{(\ell)} = \dot{\lambda}_{k,\e}^{(\ell)} \Pi_{k,\e}^{(\ell)} + \lambda_{k,\e}^{(\ell)} \dot{\Pi}_{k,\e}^{(\ell)},$$  
which can be rewritten as  
\begin{align}
    \dot{\Pi}_{k,\e}^{(\ell)} = -\left(\mathcal{P}_\e - \lambda_{k,\e}^{(\ell)}\right)^{-1} \left(\dot{\mathcal{P}}_\e - \dot{\lambda}_{k,\e}^{(\ell)}\right) \Pi_{k,\e}^{(\ell)}.\label{eq:br}
\end{align}
A common strategy is to evaluate equation \eqref{eq:br} at $\e = 0$. However, in our setting, this approach fails because the range of the operator
$$
\lim_{\e\to 0}\left.\left(\dot{\mathcal{P}}_\e - \dot{\lambda}_{k,\e}^{(\ell)} \right) \Pi_{k,\e}^{(\ell)}\right|_{\e=0}$$  
lies outside the domain of $(\mathcal{P}_\e - \lambda_{k,\e}^{(\ell)})^{-1}$. To overcome this difficulty, we use a non-trivial cancellation mechanism specific to our framework.
This allows us to obtain explicit expressions for the response terms for both the eigenvalues and eigenprojections.

Further, we observe numerically (for example as in Figure \ref{fig:respmatrixfig}(right)) that the form of the response of the eigenfunctions $\dot{F}^{(\ell)}_{k}$ is such that the overall shape of the envelope of $F^{(\ell)}_{k,\e}$ is relatively unchanged by modestly increasing $\e$ from 0.
Practically, this observation helps to explain the robustness of dominant complex eigenfunctions to various types of noise, whether they arise from models, data, or numerical techniques.


The following theorem refines Theorem \ref{thm:1} by stating the linear response of the eigenprojections, in addition to the linear and quadratic responses of the eigenvalues of $\mathcal{P}_\e$ as $\e \to 0$.

\begin{theorem}\label{thm:2}  
Under the hypotheses and notation of Theorem \ref{thm:1}, there exist complex numbers  \( \hat{\lambda}_{k}^{(1)},\ldots,\hat{\lambda}_{k}^{(N)}, \doublehat{\lambda}_{k}^{(1)},\ldots,\doublehat{\lambda}_{k}^{(N)}  
\) and functions  
$$f_k^{(1)},\ldots, f_k^{(N)}, \hat{f}_k^{(1)},\ldots, \hat{f}_k^{(N)}: \{1,\ldots, N\} \to \mathbb{C},$$ 
such that  
\begin{align}
\lambda_{k,\e}^{(\ell)} &= e^{-2 \pi i k \alpha_\ell} + \e \hat{\lambda}_{k}^{(\ell)} + \e^2 \doublehat{\lambda}_{k}^{(\ell)} + \smallO(\e^2),  \label{evalresponse}
\end{align}
and  
\begin{align}
    \Pi_{k,\e}^{(\ell)}(\cdot) =  \Pi_k^{(\ell)}(\cdot) + \e   \left( \left\langle \cdot, F_k^{(\ell)}\right\rangle \hat{F}_k^{(\ell)} +  \left\langle \cdot, \hat{F}_k^{(\ell)} \right\rangle F_k^{(\ell)} \right) + \smallO(\e),\label{evecresponse}
\end{align}
where $F_k^{(\ell)}(j,x) = f^{(\ell)}(j)e^{2\pi i k x}$ and $\hat{F}_k^{\ell}(j,x):= \hat{f}_k(j)e^{2\pi i k x}$. Also, $\langle F_k^{(\ell)}, \hat{F}_k^{(\ell)}\rangle = 0$, \rev{the term $\smallO(\e^2)\in \mathbb C$ in \eqref{evalresponse} satisfies $\smallO(\e^2)/\e^2\xrightarrow[]{\e\to 0}0$} and the term $\smallO(\e):L^2(M)\to L^2(M)$ in $\eqref{evecresponse}$ is a finite-rank operator such that $\|\smallO(\e)\|/\e \xrightarrow[]{\e\to 0}0.$




Moreover, explicit formulae for these quantities can be found as follows:
\begin{enumerate}  
\item If $S = 1$ \rev{(one band, all fibres rotate at the same speed)} or $k = 0$:
\begin{itemize}
    \item[$(a)$] $\hat{\lambda}_{k}^{(\ell)}$ and $f_k^{(\ell)}$ are, respectively, the eigenvalues and eigenvectors of $e^{2\pi i k \alpha_1} \dot{W}$;
    \item[$(b)$] $\doublehat{\lambda}_k^{(\ell)} =0$ and $\hat{f}_k^{(\ell)} =0;$ and
    \item[$(c)$] the functions $\smallO(\e^2)$ in \eqref{evalresponse} and $\smallO(\e)$ in \eqref{evecresponse} are identically equal to $0$ (defined in Proposition \ref{prop:3.3}).
\end{itemize}
    \item If $S > 1$ \rev{(multiple bands)} and $k \neq 0$:
    \begin{itemize}  
        \item[(a)] $\hat{\lambda}_{k}^{(\ell)}$ and $f_k^{(\ell)}$ are, respectively, the eigenvalues and eigenvectors of $\hat{P}_{k,\beta,L}$  (see \eqref{PkbL}). Moreover, $\arg(\hat\lambda^{(\ell)}_k)=\arg(\lambda^{(\ell)}_k)+\pi$ if $\hat\lambda^{(\ell)}_k\neq 0$.
           \item[(b)] $\hat{f}_k^{(\ell)}$ is computed in Theorem \ref{thm:fhat}; and
        \item[(c)] $\doublehat{\lambda}_{k}^{(\ell)}$ is provided in Theorem \ref{thm:hlambda}.  
    \end{itemize}  
\end{enumerate}  
\end{theorem}

\begin{example}[Response of eigenvalues and eigenfunctions at $\e=0$]
 \rev{We continue our example where $W_\e$ is given by \eqref{Weps1d}, $\beta=(\pi/20, e/7, 1/\sqrt{2})$, and  $L=(11,7,15)$.}

\rev{\emph{Eigenvalue response:} For small $\e$ the eigenvalues have a quadratic response by Theorem \ref{thm:2}, with the linear response given explicitly in Proposition \ref{eigen}.
Figure \ref{fig:respmatrixfig}(left) shows the response of the 33 eigenvalues displayed in Figure \ref{fig:pertmatrixfig}(left).
As stated in Theorem \ref{thm:2}(2)(a) and seen explicitly in Proposition \ref{eigen}, the eigenvalues are altered in magnitude only, with the argument remaining fixed at $2\pi$-integer multiples of the rotation speeds.
Figure \ref{fig:respmatrixfig}(left) demonstrates that the eigenvalues indeed move along the rays they lie on, decreasing their magnitude.
\begin{figure}
    \centering
\includegraphics[width=0.49\linewidth]{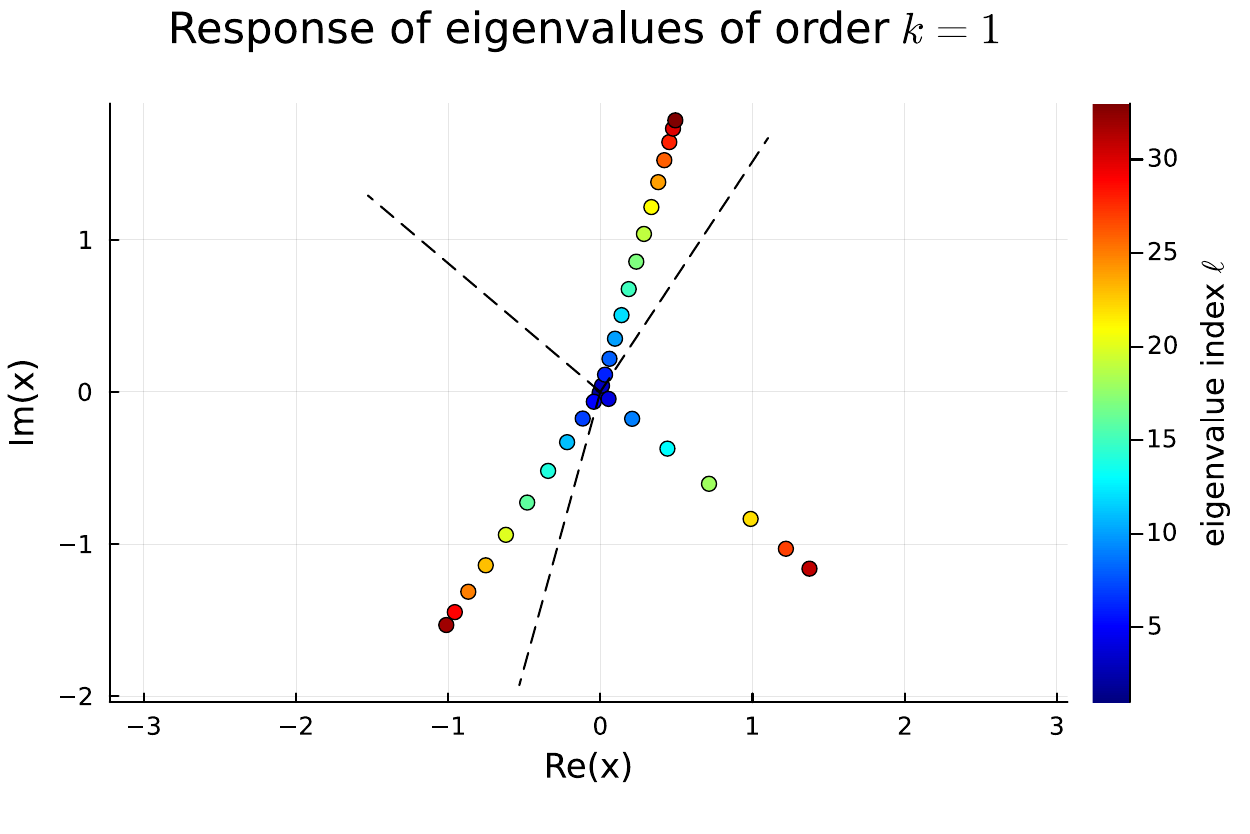}
\includegraphics[width=0.49\linewidth]{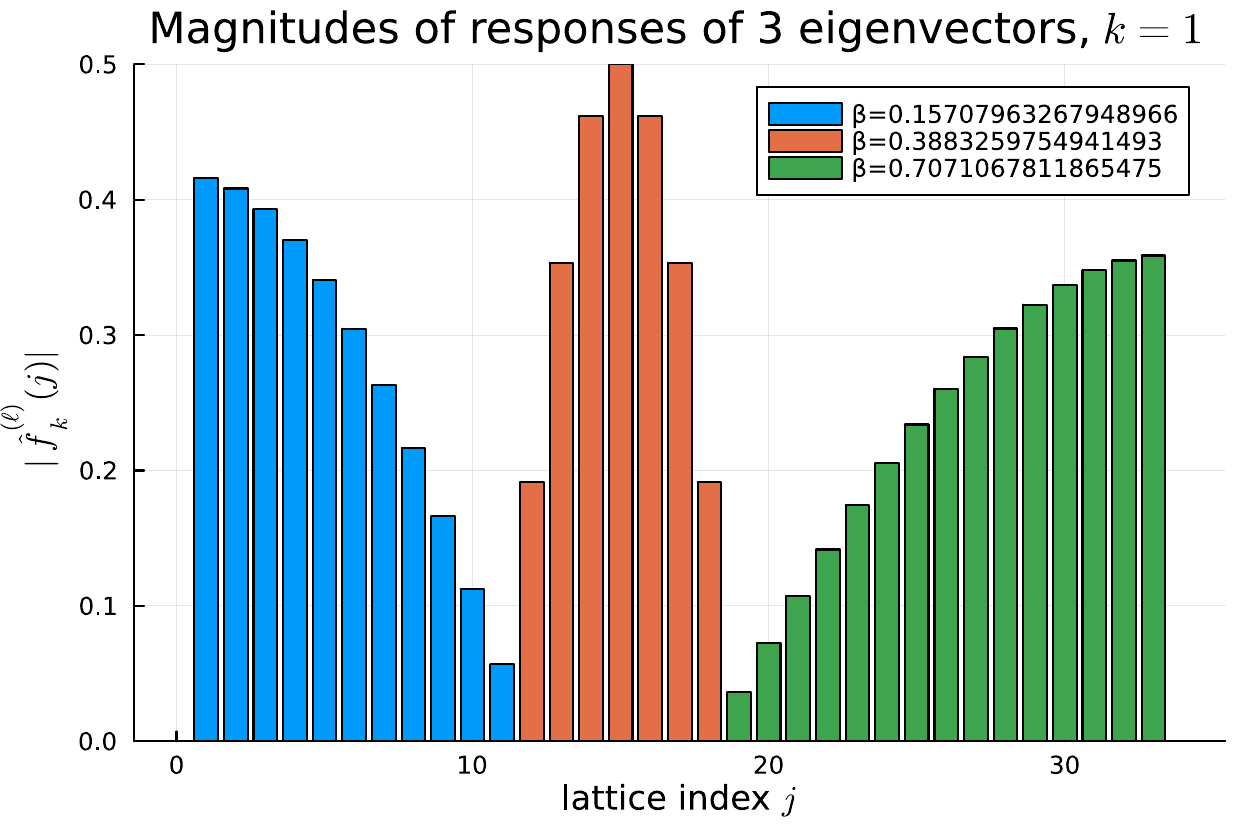}
    \caption{\textit{Left:} The small coloured disks indicate the response of the spectrum $\hat{\lambda}_{1}^{(\ell)}$, $\ell=1,\ldots,33$ of the $33\times 33$ matrix $\hat{P}_{1,\beta,L}$,  associated to the model from Figure 1, where $\beta=(\pi/20, e/7, 1/\sqrt2)$ and $L=(11,7,15)$.  Note that the responses of the 33 eigenvalues are grouped along three rays with $L_i$ responses on a ray with angle $2\pi\beta_i+\pi$ for $i=1,2,3$.  
    The dashed lines and disk colouring are identical to those in Figure \ref{fig:pertmatrixfig}(left).
\textit{Right:}
    Plot of $|\hat{f}_{1}^{(\ell)}(j)|$ vs lattice index $j$ for $\ell=1,2,4$, where the ordering of the eigenvalues $\hat{\lambda}_1^{(\ell)}$ are shown in the left panels of this figure and Figure \ref{fig:pertmatrixfig}. The three response vectors have been normalised to have unit norm and correspond to the three eigenvectors displayed in Figure \ref{fig:pertmatrixfig}(right). 
    }
    \label{fig:respmatrixfig}
\end{figure}
Those farthest from the unit circle have the greatest response.
In summary, \textit{the argument of the eigenvalues is incredibly robust to the noise level $\e$, which for practical purposes means that eigenvalues of the transfer operator are robust indicators of the periods of almost-cycles.}}

\rev{\emph{Eigenfunction response:}
Turning now to the response of the eigenvectors, in Figure \ref{fig:respmatrixfig}(right) we plot the magnitudes of responses of the three eigenvectors shown in Figure \ref{fig:pertmatrixfig}(right).
In Figure \ref{fig:respmatrixfig}(right) we can see that the responses are separately restricted to each of the three bands, and that the envelopes of the magnitudes are almost identical to those in Figure \ref{fig:pertmatrixfig}(right).
Remark \ref{rmk:resp} provides a theoretical explanation for the observed support restriction. 
Therefore, as in the case of the eigenvalues, these eigenvectors, and therefore \textit{the eigenfunctions of the transfer operator $\mathcal{P}_\e$, are extremely robust to the noise level $\e$, making them reliable indicators of the supports of distinct almost-cyclic sets in the state space of the dynamics.}}
\end{example}

\subsection{Linear response with respect to the band vectors}

In this section, we describe the linear response with respect to the rotation speed vector $\alpha \in \R^N$. To emphasise the $\alpha$ dependence on the operator $\mathcal P_\e$ we write it as $\mathcal P_{\e,\alpha}$.

\subsubsection{The case of $\varepsilon=0$.}
For $\varepsilon=0$, linear response with respect to the band vectors \rev{holds when $S=N$ (all circular fibres rotate with distinct speeds) or $k=0$, and fails when  $S\neq N$ and $k\neq 0$.}

\begin{proposition}\label{prop:linear-response-projection}
\rev{Fix $k\in\Z$ and $\ell\in\{1,\ldots,N\}$. For $\e>0$ let $\Pi_{\e,\alpha,k}^{(\ell)}$ denote the $L^2$-orthogonal projection onto $\mathrm{span}\{F_{k,\e,\alpha}\}$, where $F_{k,\e,\alpha}$ is defined in \eqref{pro:sard}. Whenever the limit exists, set $
\Pi_{k,\alpha}^{(\ell)}:=\lim_{\e\to 0}\Pi_{\e,\alpha,k}^{(\ell)}.$ Then:
\begin{enumerate}
\item[(i)] If $\alpha\in\mathbb{R}^N$ is $N$-banded (all circular fibres rotate at distinct speeds), the map $\alpha\mapsto \Pi_{k,\alpha}^{(\ell)}$ is locally constant in a neighbourhood of $\alpha$. In particular,
$
\frac{\partial}{\partial\alpha}\Pi_{k,\alpha}^{(\ell)}=0
$
at every $N$-banded vector $\alpha$. The same conclusion holds when $k=0$ and $\alpha$ is $S$-banded for any $1\le S\le N$.
\item[(ii)] If $\alpha\in\mathbb{R}^N$ is $S$-banded for some $1\leq S<N$ and $k\neq 0$, then the map $\upsilon\mapsto \Pi_{k,\upsilon}^{(\ell)}$ is discontinuous in a neighbourhood of $\alpha$. Consequently, linear response fails at $\alpha$.
\end{enumerate}}
\end{proposition}

\begin{proof}
\,
\begin{enumerate}
\item[(i)]\rev{Assume that $\alpha$ is $N$-banded. By Theorem \ref{thm:1} (2), $\Pi_{k,\alpha}^{(\ell)}$ exists and is the $L^2$-orthogonal projection onto $\mathrm{span}\{F_{k,\alpha}^{(\ell)}\}$, where $
F_{k,\alpha}^{(\ell)}(j,x)=\mathbf{1}_{\{\ell\}}(j)e^{2\pi i k x}.$ In particular, for $N$-banded parameters the range of $\Pi_{k,\alpha}^{(\ell)}$ is generated by a function that does not depend on $\alpha$. Hence $\Pi_{k,\alpha}^{(\ell)}$ itself is constant on the set of $N$-banded vectors. Since the set of $N$-banded vectors is open in $\mathbb{R}^N$, the map $\alpha\mapsto \Pi_{k,\alpha}^{(\ell)}$ is locally constant at $\alpha$, and therefore its derivative with respect to $\alpha$ vanishes there. The case $k=0$ is analogous.}

\item[(ii)] \rev{Let $k\neq 0$ and let $\alpha$ be $S$-banded with $1\leq S<N$. Since the set of $N$-banded vectors is dense in $\mathbb{R}^N$, there exists a sequence $(\upsilon_n)_{n\geq 1}$ of $N$-banded vectors such that $\upsilon_n\to\alpha$. By part (i), each $\Pi_{k,\upsilon_n}^{(\ell)}$ equals the $L^2$-orthogonal projection onto $\mathrm{span}\{\mathbf{1}_{\{\ell\}}(j)e^{2\pi i k x}\}$. 
On the other hand, Theorem \ref{thm:1} 2(b) implies that at the $S$-banded parameter $\alpha$ the limiting projection $\Pi_{k,\alpha}^{(\ell)}$ is not given by this same projection.
Therefore $\Pi_{k,\upsilon_n}^{(\ell)}\not\to \Pi_{k,\alpha}^{(\ell)}$, which proves that $\upsilon\mapsto \Pi_{k,\upsilon}^{(\ell)}$ is discontinuous at $\alpha$. Discontinuity implies failure of differentiability, so linear response fails at $\alpha$.}
\end{enumerate}
\end{proof}

\subsubsection{The case of $\e>0$.}
For $\e>0$ we have a positive result of standard linear response.
\begin{proposition}\label{prop:LRA}
 Fix $k\in \mathbb Z$, $\e>0$, and $\alpha\in \R^N$. 
   Assume that 
   $\mathcal{P}_{\e,\alpha}$ has exactly $N$ distinct eigenvalues $\lambda_{k,\e,\alpha}^{(\ell)}$, $\ell =1,\ldots,N$ associated with eigenvectors of the form $F_{k,\e,\alpha}(x,j) = f_{k,\e,\alpha}(j)e^{2\pi i k x}$.
\rev{Each of these eigenvalues enjoys linear response with respect to $\alpha$. More precisely,} there exists an open neighbourhood $A_\alpha$ of $\alpha$ in $\R^N$ such that the maps
\begin{align}
     U_{\alpha} \ni \alpha'  \mapsto \lambda_{k,\e,\alpha'}^{(\ell)} \in\mathbb C\ \text{and }   U_{\alpha} \ni \alpha' \mapsto F_{k,\e,\alpha'}(\cdot,\cdot) = f_{k,\e,\alpha'}(\cdot) e^{2\pi i k \cdot} \in\mathcal C^0(M,\mathbb C),\label{pro:sard}
\end{align} 
are differentiable.
Moreover,  one has
\begin{align}\frac{\partial}{\partial \alpha}f_{k,\e,\alpha} = \left(D_{k,\alpha} W_\e - \lambda_{k,\e,\alpha} \right)^{-1} \left(\frac{\partial }{\partial \alpha}D_{k,\alpha} W_\e - \frac{\partial}{\partial \alpha} \lambda_{k,\e,\alpha}\right) f_{k,\e,\alpha}. \label{eq:scarpeli}
\end{align}
\end{proposition}
\begin{proof}
See Section \ref{sec:alpharesp}.    
\end{proof}

We remark that by Theorem \ref{thm:1}(1), for sufficiently small $\varepsilon$ one automatically obtains the assumption in Proposition \ref{prop:LRA} that 
$\mathcal{P}_{\e,\alpha}$ has exactly $N$ eigenvalues.

\section{Proofs of Proposition \ref{prop:1}, Theorem \ref{thm:1}\texorpdfstring{$(a)$}{(a)}, Theorem \ref{thm:1}(2)\texorpdfstring{$(a)$}{(a)}, and Theorem \ref{thm:2}(1)} 
\label{sec:3}



\subsection{Proof of Proposition \ref{prop:1}}
\label{sec:3.1}
We start by proving results that are valid for any  $\alpha\in \R^N$. In  future sections, we restrict our analysis to the vector $\alpha$ being an $S$-band. 

\subsubsection{Proof of Proposition \ref{prop:1}(1)}
Lemma \ref{lem:eform} asserts that eigenfunctions of 
$\mathcal P_{\e}$ corresponding to a fixed Fourier index $k$ take the form of a
harmonic function defined along {each} circular fibre, modulated by a function that is solely dependent on the fibres.
Recall that $D_{k,\alpha}$ is defined in \eqref{Dkalpha}.

\begin{lemma}
\label{lem:eform}
A function $F\in L^2(M)$ is a generalised eigenfunction of  $\mathcal P_{\e}$ with eigenvalue $\lambda$ if and only if ${F(j,x) = f_k(j) e^{2 \pi i k x}}$ for some $k\in\mathbb Z$, where (identifying the function $f_k$ with the vector $(f_k(1),\ldots, f_k(N)) \in \mathbb C^N$) $f_k$ is a generalised eigenvector of $D_{k,\alpha} W_\e$ with eigenvalue $\lambda$.
\end{lemma}
\begin{proof}
Given $f\in L^2(M)$, by employing Fourier series, we can express 
\begin{equation}
    \label{fourierform}
F(j,x) = \sum_{k\in\mathbb Z} f_k(j)e^{2 \pi i k x},\ \text{where }f_k: \{1,\ldots,N\}\to \mathbb C.
\end{equation}
By \eqref{Pbase} and \eqref{skewproduct} with $\delta=0$ for every $(j,x)\in M$ 
\begin{align}
    &\mathcal P_{\e} F(j,x) =\sum_{k\in\Z} 
\mathcal P_\e \left(f_k(\cdot) e^{2\pi i k \cdot} \right)(j,x)\nonumber\\
\label{eigeneqn}    &=\sum_{k\in\mathbb Z} e^{-2 \pi i k \alpha_j} \left(W_\e f_k\right)_j  e^{2 \pi i k x}.
\end{align}
Suppose that $\mathcal{P}_\varepsilon F=\lambda F$.
By equating individual Fourier mode coefficients in \eqref{eigeneqn} we obtain $D_{k,\alpha}W_\varepsilon f_k=\lambda f_k$.
Similarly, if $D_{k,\alpha}W_\varepsilon f_k=\lambda f_k$ holds, the ansatz \eqref{fourierform} provides an eigenfunction $F$ with eigenvalue $\lambda$.
The case where $F$ is a generalised eigenfunction of $F$ follows similarly.
\end{proof}

\subsubsection{Proof of Proposition \ref{prop:1}(2)}
\begin{lemma} \label{prop:xama}

Given $k\in \mathbb Z$, let $\{\lambda_{k,\e}^{(1)}, \ldots, \lambda_{k,\e}^{(N)}\}:=\sigma\left(D_{k,\alpha}  W_\e\right)$. Then, after properly relabelling the elements of $\sigma\left(D_{k,\alpha}  W_\e\right)$, it follows that for every $\e>0$ sufficiently small
    $$\left|\lambda_{k,\e}^{(\ell)} -  e^{-2 \pi i k \alpha_\ell}\right| \leq 2 \max_{1\leq r \leq N}(1- W_{\e,r, r})\ \text{for every}\ \ell\in\{1,\ldots, N\},$$
    and
     \rev{$$\left|\mathrm{arg}\left(\lambda_{k,\e}^{(\ell)}\right) + 2\pi i k \alpha_\ell \ (\mathrm{mod}\, 2\pi)\right| \leq 2\pi  \max_{1\leq r \leq N}(1- W_{\e,r, r}) \ \text{for every}\ \ell\in\{1,\ldots, N\}.$$}
\end{lemma}
\begin{proof}

\rev{For each $\ell\in\{1,\ldots,N\}$, consider the Gershgorin disks
$$
\Delta_{\ell}
:=\left\{z\in\mathbb C:\left|z-(D_{k,\alpha}W_\e)_{\ell\ell}\right|
\leq\sum_{\ell'\neq\ell}\left|(D_{k,\alpha}W_\e)_{\ell\ell'}\right|\right\}.
$$
Since $D_{k,\alpha}$ is diagonal, for all $\ell,\ell'$,
$$
(D_{k,\alpha}W_\e)_{\ell\ell'}=(D_{k,\alpha})_{\ell\ell}\,W_{\e,\ell,\ell'}
=e^{-2\pi i k\alpha_\ell}W_{\e,\ell,\ell'}.
$$
Hence $
\sum_{\ell'\neq\ell}\left|(D_{k,\alpha}W_\e)_{\ell\ell'}\right|
=\sum_{\ell'\neq\ell}W_{\e,\ell,\ell'}=1-W_{\e,\ell,\ell}$.
Therefore
\begin{align*}
\Delta_{\ell}
&=\left\{z\in\mathbb C:\left|z-e^{-2\pi i k\alpha_\ell}W_{\e,\ell,\ell}\right|
\leq1-W_{\e,\ell,\ell}\right\}\\
&\subset\left\{z\in\mathbb C:\left|z-e^{-2\pi i k\alpha_\ell}\right|
\leq2\max_{1\leq r\leq N}(1-W_{\e,r,r})\right\}
=:\widetilde{\Delta}_\ell.
\end{align*}}
\rev{Take $\e_0>0$ sufficiently small such that the number of connected components of
$\bigcup_{\ell=1}^N\widetilde{\Delta}_\ell$ is constant for every $0<\e<\e_0$.
By the Gershgorin circle theorem \cite[Theorem 6.11]{MAnalysis}, after relabelling the eigenvalues we obtain
$$
\left|\lambda_{k,\e}^{(\ell)}-e^{-2\pi i k\alpha_\ell}\right|
\leq2\max_{1\leq r \leq N}(1-W_{\e,r,r})
\ \text{for every}\ \ell\in\{1,\ldots,N\},
$$
which proves the first inequality.}

\rev{We now prove the second inequality.
Fix $\ell\in\{1,\ldots,N\}$. By the reverse triangle inequality,
\begin{align}
\left|1-\left|\lambda_{k,\e}^{(\ell)}\right|\right|
=\left|\left|e^{-2\pi i  k\alpha_\ell }\right|-\left|\lambda_{k,\e}^{(\ell)}\right|\right|
\leq \left|\lambda_{k,\e}^{(\ell)}-e^{-2\pi i k\alpha_\ell}\right|.\label{eq:rti}
\end{align}}
\rev{Then, by the triangle inequality and \eqref{eq:rti},
\begin{align*}
\left|e^{i\arg(\lambda_{k,\e}^{(\ell)})}-e^{-2\pi i k\alpha_\ell}\right|
&=\left|\frac{\lambda_{k,\e}^{(\ell)}}{\left|\lambda_{k,\e}^{(\ell)}\right|}-e^{-2\pi i k\alpha_\ell}\right|\leq\left|\frac{\lambda_{k,\e}^{(\ell)}}{\left|\lambda_{k,\e}^{(\ell)}\right|}-\lambda_{k,\e}^{(\ell)}\right|
+\left|\lambda_{k,\e}^{(\ell)}-e^{-2\pi i k\alpha_\ell}\right|\\
&=\left|1-\left|\lambda_{k,\e}^{(\ell)}\right|\right|
+\left|\lambda_{k,\e}^{(\ell)}-e^{-2\pi i k\alpha_\ell}\right|\leq2\left|\lambda_{k,\e}^{(\ell)}-e^{-2\pi i k\alpha_\ell}\right|\\
&\leq4\max_{1\leq r\leq N}(1-W_{\e,r,r}).
\end{align*}
}
\rev{Finally, let
$\Delta:= \left|\arg(\lambda_{k,\e}^{(\ell)})-(-2\pi k\alpha_\ell)\ (\mathrm{mod}\ 2\pi)\right| \in[0,\pi]$ be the smallest angular distance between $\mathrm{arg}(\lambda_{k,\e}^{(\ell)})$ and $-2\pi i k \alpha_\ell$. On the unit circle,
$$
\left|e^{i\arg(\lambda_{k,\e}^{(\ell)})}-e^{-2\pi i k\alpha_\ell}\right| =2\sin(\Delta/2)\geq\frac{2}{\pi}\Delta,
$$
and therefore
$$
\left|\arg(\lambda_{k,\e}^{(\ell)})-(-2\pi k\alpha_\ell)\ (\mathrm{mod}\ 2\pi)\right| = \Delta\leq\frac{\pi}{2}\left|e^{i\arg(\lambda_{k,\e}^{(\ell)})}-e^{-2\pi i k\alpha_\ell}\right|
\leq2\pi\max_{1\leq r\leq N}\left(1-W_{\e,r,r}\right).
$$}

\end{proof}



\subsection{Distinct eigenvalues of \texorpdfstring{$\mathcal{P}_\e$}{Pe} for small \texorpdfstring{$\varepsilon>0$}{e>0} }
\label{sec:3.2}
The proofs of Theorem \ref{thm:1}(1), Theorem \ref{thm:1}(2)(i), and Theorem \ref{thm:2}(1) all rely on the eigenspaces of $\mathcal{P}_\e$ being one-dimensional.
Recalling the correspondence between the spectrum of $\mathcal P_\e$ and $D_{k,\alpha}W_\e$ established in Lemma \ref{lem:eform}, we adopt the following simplified notation in this and all subsequent sections. {Fixing} $\beta$, $L$, and  $k\in\mathbb{Z}$, for every $\e>0$ we write \begin{align}
        D_\beta := D_{k,\beta,L}\ \text{and }P_{\e,\beta} := D_\beta W_\e.
        \label{eq:df}
    \end{align}
For $v,w\in\mathbb C^N$ we denote the inner product between $v$ and $w$ as
$\langle v , w\rangle =\sum_{j=1}^{N} v_j \overline{w_j}$, and for every $s\in\{1,\dots, N\}$ we set 
    \begin{align}
        V_s &:= \{v\in \mathbb C^N: v_j = 0\ \text{if }j\leq N_{s-1}\ \text{or }j>N_{s}\}\nonumber \\
        &= \{v\in \mathbb C^N: v_j = 0\ \text{for every } j\in \{1,\ldots, N\}\setminus  B_s\}.\label{eq:Vi}
    \end{align}
The following three results establish  that for every 
$\beta \in \Gamma$ the matrix $P_{\e,\beta}$ is diagonalizable for every $\e>0$ sufficiently small.


\begin{theorem}
    \label{thm:gamma} 
Under the hypotheses of Theorem \ref{thm:1}
there exists $\gamma_k>0$ such that for every $0<\e<\gamma_k$ the matrix $P_{\e,\beta}$ has $N$ distinct eigenvalues. 
\end{theorem}

Before proving the above theorem, we establish three preparatory lemmas.

\begin{lemma}\label{lem:zero}
Fix $s\in\{1,\ldots, S\}.$ For each $\e>0$ choose a sequence of unit-norm vectors $\{f_\e\}_{\e >0}\subset \mathbb C^N$ such that for every $\e>0$,
$$P_{\e,\beta} f_\e = \lambda_\e f_\e\ \text{and }\lambda_\e\xrightarrow[]{\e\to 0}e^{-2\pi i k \beta_s}.$$ 
Then, $ \left(\displaystyle f_{\e}\right)_j\xrightarrow[]{\e\to 0} 0$ for every $j\in\{1,\ldots,N\}\setminus B_{s}.$ 


\end{lemma}

\begin{proof}
 Since $\|\displaystyle f_\e\|=1$ and for every $\e$ small enough there exists a converging subsequence $\displaystyle \{f_{\e_r}\}_{r\in\mathbb N}\subset \{f_{\e}\}_{\e>0}$ such that $\displaystyle f_{\e_r} \xrightarrow[]{\e_r\to 0} f_*$.  It follows that
 $$D_{\beta} f_* = \lim_{k\to\infty} D_{\beta} W_{\e_r}  f_{\e_r} = \lambda_{\e_r} f_{\e_r} = e^{-2\pi i k \beta_s} f_*\ \text{and therefore}\ \left(D_\beta - e^{-2\pi i k\beta_s}\right)f_*=0.$$
 Observe that for each $j\in B_t$ for some $t\in\{1,\ldots,N\}$ we obtain that
$$ \left(e^{-2\pi i k \beta_t} - e^{-2\pi i k \beta_s}\right)(f_*)_j =0$$
  Since $\beta\in \Gamma$, then $\min_{t\neq s} |e^{-2\pi i k \beta_{s_1}} - e^{-2\pi i k \beta_{s_2}}|>0$ which implies that  $\displaystyle f_* =0$ for every $j\in\{1,\ldots,N\}\setminus B_s.$  Hence, for every subsequence of $\{\displaystyle f_\e\}_{\e>0}$, it is possible to construct a subsubsequence such that all coordinates $j\in\{1,\ldots, N\}\setminus B_s$ 
 converge to $0$ as $\e\to 0$, which implies the desired result.
\end{proof}

\begin{lemma}
    \label{lem:degenerated}
Assume that there exist vectors $f,g \in \mathbb C^N$, $\beta\in\R^S$ and $\e>0$ such that
$$P_{\e,\beta}  f= \lambda f\ \text{and }P_{\e,\beta}  g = \lambda g + f.$$
Then, $\langle f , D_{\beta} \overline{f}\rangle =0.$
\end{lemma}
\begin{proof}
Since $P_{\e,\beta} g = \lambda g +  f$, then $\overline {D_{\beta} W_\e  g} = \overline{\lambda g + f}.$ Therefore $W_\e  \overline{g} = \overline{\lambda} D_{\beta} \overline{g} +   D_{\beta}  \overline{f}.$ Hence,
\begin{align*}
    \langle f, D_{\beta}  \overline{g}\rangle &=\frac{1}{\lambda} \langle  P_{\e,\beta} f , D_\beta \overline{g}\rangle = \frac{1}{\lambda}  \langle D_{\beta} W_\e  f, D_{\beta} \overline{g}\rangle =  \frac{1}{\lambda}  \langle  f, W_\e \overline{g}\rangle\\
    &=  \frac{1}{\lambda}  \langle  f, \overline{\lambda} D_{\beta}\overline{g} +   D_{\beta}  \overline{f}\rangle =    \langle  f,  D_{\beta}  \overline{g}\rangle +  \frac{1}{\lambda} \langle  f,  D_{\beta}  \overline{f}\rangle.
\end{align*}
Thus, $\langle  f,  D_{\beta} \overline{f}\rangle = 0.$
\end{proof}

\begin{lemma}\label{lem:46}
Let $\beta \in \R^S$ and $\{f_\e\}_{\e>0}\subset \mathbb C^N$ be a sequence of unit-norm  vectors such that for every $\e>0$ small
$$P_{\e,\beta}  f_\e = \lambda_\e f_\e \ \text{and }\lambda_\e \xrightarrow[]{\e \to 0} e^{-2\pi i k \beta_s}\ \text{for some }s\in\{1,\ldots, S\}. $$
Then, the accumulation points of $\{\displaystyle f_{\e}\}_{\e>0}$ as $\e\to 0$ are eigenvectors of the matrix $\hat{P}_{k,\beta,L}$ (defined in \ref{PkbL}). 
\end{lemma}
\begin{proof}
Consider a converging subsequence $\left\{\displaystyle f_{\e_r}\right\}_{r\in\mathbb N}\subset\left\{\displaystyle f_{\e}\right\}_{\e>0}$ such that $\displaystyle f_{\e_r} \xrightarrow[]{\e_r\to 0} \displaystyle f_*$.  From Lemma \ref{prop:xama}, we can choose a subsubsequence $\{{\e_r'}\}_{r'\in\mathbb N} \subset \{{\e_r}\}_{r\in\mathbb N}$ such that
$$\frac{\lambda_{\e'_{r}} - e^{-2\pi i k \beta_s} }{\e_r'} \xrightarrow[]{\e_r' \to 0} \lambda_*.$$
Moreover, from Lemma \ref{lem:zero} we have that $\displaystyle f_*\in V_s$ (see \eqref{eq:Vi}). Since
$$ \lambda_\e f_\e  = P_{\e,\beta} f_\e= D_{\beta} W_\e f_\e =  D_{\beta} \left(\mathrm{Id}+\e \dot{W}\right) f_\e .$$
We obtain that $\displaystyle\dot{W} f_\e = \frac{1}{\e}  \left( \lambda_\e D_{k,-\beta,L}  f_\e  - f_\e \right)$ for each $\e>0.$ Let $\pi_s:\C^N\to V_s$ the orthogonal projection from $\C^N$ into $V_s$. It follows that 
\begin{align*}
   \pi_s \dot{W} f_\e &= \frac{1}{\e}  \pi_s \left( \lambda_\e D_{k,-\beta,L}  f_\e  - f_\e \right) = \frac{e^{2\pi i k \beta_s } \lambda_\e -1}{\e}   \pi_s f_\e.
\end{align*}
By taking $\e_r'\to 0$ and observing that $\displaystyle \pi_s f_* =f_*,$ we obtain that $\displaystyle \pi_s \dot{W} f_* = \lambda_*f_*.$ Thus, $\displaystyle \hat{W}_{L} f_* = \lambda_*f_*$ (see \eqref{What}). Therefore, we obtain that
$$\hat{P}_{k,\beta,L} f_* = D_{\beta} \hat{W}_{L} f_* = e^{2\pi i k \beta_s} \lambda_*f_*.$$

\end{proof}

\begin{proof}[Proof of Theorem \ref{thm:gamma}] Assume for a  contradiction that for every $\e_0>0$ there exists $0<\e <\e_0$ such that $D_{\beta} W_\e$ does not have $N$ distinct eigenvalues. Then, there exists a sequence of real numbers  $\e_n\xrightarrow[]{}0$, as $n\to\infty,$ and sequences of vectors $\{f_{\e_n}\}_{n\in\mathbb N},\{g_{\e_n}\}_{n\in\mathbb N}\subset\C^N$ 
such that for every $n\in\mathbb N$ one has $\|f_{\e_n}\|=1$ and either
\begin{enumerate}
    \item[$(1)$] $P_{\e,\beta} g_{\e_n} = \lambda_\e g_{\e_n} + f_{\e_n}$, or
    \item[$(2)$] $f_{\e_n}$ and $g_{\e_n}$ are linearly independent and
    $$P_{\e_n,\beta}  f_{\e_n} = \lambda_{\e_n} f_{\e_n}\ \text{and } P_{\e_n,\beta} g_{\e_n}= \lambda_{\e_n} g_{\e_n}. $$
\end{enumerate}

We first consider scenario $(1)$. By passing through a subsequence, if necessary, we can assume that there exists $f_*$ and $\lambda_*$ such that
$$f_{\e_n}\xrightarrow[]{n\to\infty} f_*\ \text{and }\lambda_{\e_n}\xrightarrow[]{n\to\infty} e^{-2\pi i k \beta_s} \ \text{for some }s \in \{1,\ldots, S\}. $$

From Lemmas \ref{lem:zero},  \ref{lem:degenerated} and \ref{lem:46} respectively, we obtain that
\begin{itemize}
    \item[$(a)$] $f_* \in V_{s}$ (see \eqref{eq:Vi}) for some $s\in\{1,\ldots,S\}$;
    \item[$(b)$] $\langle f_{*}, D_{\beta} \overline{f_{*}} \rangle =  \lim_{n\to \infty}  \langle f_{\e_n}, D_{\beta}  \overline{f_{\e_n}} \rangle = 0;$
    \item[$(c)$] $f_*$ is an eigenvector of $\hat{P}_{k,\beta,L}$.
\end{itemize}
From the block structure of $\hat{P}_{k,\beta,L}$ we have that $f_*$ that is an eigenvector of $\hat{W}_{L}$ (see \eqref{What}). Since $f_*\in V_s$ and all the blocks of $\hat{W}_{L}$ are symmetric real matrices with distinct (real) eigenvalues, we obtain that $f_*$ and $\overline{f_*}$ are linearly dependent, which implies that
$\langle f_{*}, D_{\beta} \overline{f_{*}} \rangle \neq 0$, contradicting item $(b)$ above.

Now we assume $(2)$. Define
$$h_{\e_n} := \frac{1}{\left\|g_{\e_n} - \langle g_{\e_n}, f_{\e_n} \rangle f_{\e_n} \right\|} \left(g_{\e_n} - \langle g_{\e_n}, f_{\e_n} \rangle f_{\e_n} \right).$$
By passing, if necessary, through a subsequence, we can assume $f_*,h_*\in \C^N$ that
$$f_{\e_n}\xrightarrow[]{n\to\infty} f_*,\ h_{\e_n}\xrightarrow[]{n\to\infty}h_*\ \text{and }\frac{\lambda_{\e_n} - e^{-2\pi i k \beta_s}}{\e_n}\xrightarrow[]{n\to\infty} \lambda_* \ \text{for some }s \in \{1,\ldots, S\}.$$
From Lemma \ref{lem:zero} we have that  $f_*,h_*\in V_s\setminus\{0\}$ and by construction $\langle f_*,h_*\rangle = 0$. 
From Lemma \ref{lem:46} we obtain that
$$\hat{P}_{k,\beta,L}  f_* = \lambda_* f_*\ \text{and }\hat{P}_{k,\beta,L}  h_* = \lambda_* h_*,$$
and therefore
\begin{align}\label{eq:contradiction}
  \pi_s \hat{W}_{L} f_* = e^{2\pi k \beta_s } \lambda_*  f_*\ \text{and } \pi_s \hat{W}_{L} h_* =  e^{2\pi k \beta_s } \lambda_*  h_*,  
\end{align}
where $\pi_s :\mathbb C^N \to V_s$ is the $\mathbb C^N$-orthogonal projection onto $V_s$. Since $W_\e$ is an $L$-admissible family of matrices, Definition \ref{Ladmiss}(3) implies that $\pi_s \hat{W}_{L}:V_s \to V_s$ has $N$ distinct eigenvalues, which \rev{contradicts} \eqref{eq:contradiction} since $\langle f_*,h_*\rangle = 0$.
\end{proof}



\subsection{Proof of Theorem \ref{thm:1}(2)\texorpdfstring{$(a)$}{(a)} and Theorem \ref{thm:2}(1)}

\label{sec:3.3}


\begin{proposition}\label{prop:3.3}
\noindent  Under the assumptions and notation of Theorem \ref{thm:1}, if either $S=1$ or $k=0$, then for each $\ell\in\{1,\ldots,N\},$ $$\lambda_{k,\e}^{(\ell)} =e^{-2 \pi i k \beta_1} +\e \hat{\lambda}_k^{(\ell)}\ \text{and }F_{k,\e}^{(\ell)}(j,x) = f_{k}^{(\ell)}(j)e^{2\pi i k x}.$$
Moreover, $\{f_k^{(1)},\ldots, f_k^{(N)}\}$ is an orthonormal basis of $e^{-2\pi i k\alpha_1}\dot{W}$, and for each $\ell\in\{1,\ldots, N\}$, $\hat{\lambda}_k^{(\ell)}$ is an eigenvalue of  $e^{-2\pi i k\alpha_1}\dot{W}$ associated with $f_k^{(\ell)}$. In particular, $\Pi_{k,\e}^{(\ell)} = \Pi_{k}^{(\ell)}$ for every $\e>0,$ where $\Pi_{k}^{(\ell)}$ is the $L^2$-orthogonal projection onto  $\s\{F_k^{(\ell)}\}$, where $F(j,x)= f_k^{(\ell)}(j) e^{2\pi i k x}$. 
\end{proposition}

 \begin{proof}
From Lemma \ref{lem:eform} it follows that $ \mathcal P_\e F_\e =\lambda_\e F_\e$ if and only if $F(j,x) = f(j)e^{2\pi i k x}$ and $D_{\beta}W_\e f_\e= \lambda_\e f_\e.$
Since either $S=1$ or $k=0$ we obtain that $D_{\beta} = e^{-2 \pi i k \beta_1} \mathrm{Id}_N.$ Therefore
$$D_{\beta} W_\e= e^{-2 \pi i k \beta_1}W_\e = e^{-2 \pi i k \beta_1} \left(\mathrm{Id}_{N} + \e  \dot{W} \right) .$$

Since either $S=1$ or $k=0$ we obtain that $e^{-2 \pi i k \beta_1}( \mathrm{Id}_{N} + \e   \dot{W}) f_\e = \lambda_\e f_\e$ if and only if
$$  e^{-2 \pi i k \alpha_1} \dot{W} f_\e  =  \frac{ \lambda_\e -e^{-2 \pi i k \beta_1}}{\e} f_\e.$$
Hence, for every $\ell\in\{1,\ldots, N\}$
$$\lambda^{(\ell)}_{k,\e}= e^{-2 \pi i k \beta_1} + \e \hat{\lambda}^{(\ell)} \text{ and}\ f^{(\ell)}_{k,\e} =f^{(\ell)}_k.$$
Observe that $f_k^{(\ell)}$ is an orthonormal basis since $W_\e$ is an $L$-admissible family of matrices.

\end{proof}

\section[Proof of Theorem~\ref{thm:1}(2)(b) and Theorem~\ref{thm:2}: Linear component of (15) and Part~(2)(b)]
{Proof of Theorem~\ref{thm:1}(2)\texorpdfstring{$(b)$}{(b)} and Theorem~\ref{thm:2}: Linear component of \texorpdfstring{$(15)$}{(15)} and Part~(2)\texorpdfstring{$(b)$}{(b)}}

\label{sec:4}

To simplify notation we fix from now on $N, S\in\mathbb N$ with $1<S\leq N$, a cluster length vector $L\in \R^S$, an $L$-admissible family of matrices $W_\e = \mathrm{Id}+\e \dot{W}$, and $k\in \mathbb N$. Moreover, given a function $h:\{1,\ldots,N\}\to \C,$ we denote $(h)_j:=h(j)$ for every $j\in\{1,\ldots,N\},$ and we make no distinction between the vector $((h)_1,\ldots,(h)_N)$ and the function $h$. Finally, we say $(h)_j = 0$ if $h(j)$ is undefined, i.e. when $j\in\mathbb Z\setminus\{1,\ldots,N\}.$

Below, we recall the definition of the complex projective space $\mathbb C\mathbb P^{N-1}$ which is used throughout the rest of the proofs of Theorems \ref{thm:1} and \ref{thm:2}.
\begin{definition}[Complex Projective Space] \label{def:proj} Given $N\in\mathbb N$, we denote
$$\mathbb C \mathbb P^{N-1} = \mathbb C^N\setminus\{0\}/ \sim, $$
where $\sim$ is the equivalence relation in $\mathbb C^N\setminus\{0\}$
$$v\sim w\ \text{if and only if }v=\lambda w\ \text{for some }\lambda\in \mathbb C.$$
We induce a differential structure in $\mathbb C\mathbb P^N$ via the quotient map
\begin{align*}
 \pi : \mathbb C^n\setminus \{0\}&\to\mathbb C\mathbb P^N\\
 v&\mapsto [v] 
\end{align*}
where $[v] \in \mathbb C^n$ is the equivalence class of $v$ in $\sim$ (see \cite[Chapter 1, page 31]{Lee}). Given $w,v\in \mathbb C \mathbb P^{N-1}$ we define the distance between $w$ and $v$ in  $\mathbb C \mathbb P^{N-1}$ as
$$\mathrm{dist}_{\mathbb C\mathbb P^{N-1}} (v,w) = \min_{\theta \in [0,2\pi) } \left\| \frac{v}{\|v\|} - e^{i \theta} \frac{w}{\|w\|} \right\|.$$
    
\end{definition}

\begin{theorem}\label{thm:3}
Let $\beta \in \Gamma$, for each $\e>0$ small enough, let $\{\displaystyle f^{(\ell)}_{\e}\}_{\ell=1}^{N}\subset \mathbb C^N$ denote a unit-norm eigenbasis of $P_{\e,\beta}$.
Then, 
\begin{enumerate}
    \item[$(i)$] there exists an orthonormal eigenbasis $\{f^{(1)},\ldots, f^{(N)}\}\subset \mathbb C^N$ of
\begin{equation}
    \label{eq:Tb}
\hat{P}_{\beta} := D_\beta \hat{W}_{L},
\end{equation}
where $\hat{W}_{L}$ is defined in \eqref{What} and $D_\beta$ in \eqref{eq:df}, such that for every $j\in\{1,\ldots,N\},$ one has 
\begin{align}\displaystyle [f^{(\ell)}_{\e}] \xrightarrow[]{\e\to 0} [f^{(\ell)}]\ \text{in }\mathbb C\mathbb P^{N-1}\ \text{for every } \ell\in\{1,\ldots,N\}.\label{eq:feconv}
\end{align}
\item[$(ii)$] For $s\in\{1,\ldots,N\}$ and $\ell \in B_s$, let $\hat\lambda^{(\ell)}$ be the  eigenvalue associated to the eigenvector $f^{(\ell)}$ of $\hat{P}_\beta$. Then
\begin{align}
    \lambda_\e^{(\ell)} = e^{-2\pi i k \beta_s} + \e \hat{\lambda}^{(\ell)} + \smallO(\e).\label{eq:lambda1staprox}
\end{align}
\end{enumerate}
Moreover, $\arg(\hat\lambda^{(\ell)})=\arg(\lambda^{(\ell)})+\pi$ if $|\hat \lambda^{(\ell)}|>0$.
\end{theorem}

\subsection{Proof of Theorem \ref{thm:3}}\label{sec:41T}


The proof of Theorem \ref{thm:3} is lengthy, \rev{so we establish two preparatory lemmas before the proof}. 
For every $\e>0$ small enough we write $\{\displaystyle f^{(\ell)}_{\e}\}_{\ell=1}^{N}\subset \mathbb C^N$ as a unit-norm eigenbasis of $P_{\e,\beta}$, and by appealing to standard matrix perturbation theory \cite[Chapter Two \S{1}]{Kato} we obtain that for every $\ell\in\{1,\ldots,N\},$ the map
    $$ \e \in(0,\gamma_k) \mapsto \displaystyle f_\e^{(\ell)}\in \mathbb C^N$$
    is $\mathcal C^\infty$. We define $\lambda_{\e}^{(\ell)}$ as the eigenvalue associated to the eigenvector $\displaystyle f^{(\ell)}_{\e}$. Recall that
    $$\lambda_{\e}^{(\ell)}\to e^{-2\pi i k \beta_s}\ \text{if }\ell \in B_s\ \text{as }\e\to 0.$$
We mention that standard perturbation theory asserts \cite[Chapter Two \S{1}]{Kato} that the above map $\e \in (0,\gamma_k)\mapsto \lambda_{\e}^{(\ell)} \in\mathbb C$  is $\mathcal C^\infty$. Finally, for every $\ell\in \{1,\ldots, N\},$ $\Pi_\e^{(\ell)}: L^2(M)\to \mathrm{span}\{\displaystyle F_{\e}^{(\ell)}\}$ denotes the $L^2$-orthogonal projection onto the $\s\{\displaystyle F_\e^{(\ell)}\}$ where $\displaystyle F_{\e}(j,x):= f_{\e}^{(\ell)} e^{2\pi i k x}.$

\begin{lemma} Let $1\leq \ell\neq m\leq N$, then $\langle \displaystyle f^{(\ell)}_{\e}, D_\beta \overline{\displaystyle \displaystyle f^{(m)}_{\e}}\rangle=0$ for every $\ell\neq m$ and $\e>0$ small.\label{lem:ld} 
\end{lemma}
\begin{proof} 
From Lemma \ref{prop:xama} $ \lambda_{\e}^{(\ell)} \neq 0$ for every $\ell\in \{1,\ldots, N\}$ and $\e >0$ sufficiently small. Observe that 
\begin{align} \displaystyle  \langle f^{(\ell)}_{\e}, D_\beta \overline{\displaystyle \displaystyle f_\e^{(m)}}\rangle &= \frac{1}{\lambda_{\e}^{(\ell)}}\langle P_{\e,\beta} f^{(\ell)}_{\e}, D_\beta \overline{\displaystyle \displaystyle f^{(m)}_{\e}}\rangle = \frac{1}{\lambda_{\e}^{(\ell)}}\langle D_\beta W_\e f^{(\ell)}_{\e}, D_\beta \overline{\displaystyle \displaystyle f^{(m)}_{\e}}\rangle \nonumber\\ 
\displaystyle&= \frac{1}{\lambda_{\e}^{(\ell)}} \langle W_\e  f^{(\ell)}_{\e}, \overline{\displaystyle \displaystyle f^{(m)}_{\e}}\rangle= \frac{1}{\lambda_{\e}^{(\ell)}} \langle   f^{(\ell)}_{\e}, W_\e \overline{\displaystyle \displaystyle f^{(m)}_{\e}}\rangle.\label{eq:idk1} \end{align}

Since $ \displaystyle D_\beta W_\e \displaystyle f^{(m)}_{\e} = \lambda_{\e}^{(m)} \displaystyle \displaystyle f^{(m)}_{\e}$, it follows that $ W_\e \displaystyle f^{(m)}_{\e} = \lambda_{\e}^{(m)} D_{-\beta} \displaystyle f^{(m)}_{\e}$ and therefore $W_\e  \overline{\displaystyle \displaystyle f^{(m)}_{\e}} = \overline{\lambda_{\e}^{(m)}} D_\beta \overline{\displaystyle \displaystyle f^{(m)}_{\e}} $. 
Combining the last equality with \eqref{eq:idk1}, we obtain 
\begin{align*} \langle f^{(\ell)}_{\e}, D_\beta \overline{\displaystyle \displaystyle f^{(m)}_{\e}}\rangle &= \frac{1}{\lambda_{\e}^{(\ell)}} \langle   f^{(\ell)}_{\e}, W_\e \overline{\displaystyle \displaystyle f^{(m)}_{\e}}\rangle =  \frac{1}{\lambda_{\e}^{(\ell)}} \langle   f^{(\ell)}_{\e},  \overline{\lambda_{\e}^{(m)}} D_\beta \overline{\displaystyle \displaystyle f^{(m)}_{\e}}\rangle = \frac{\lambda_{\e}^{(m)}}{\lambda_{\e}^{(\ell)}}\langle f^{(\ell)}_{\e}, D_\beta \overline{\displaystyle \displaystyle f^{(m)}_{\e}}\rangle. \end{align*} 
Finally, from Theorem \ref{thm:gamma} we obtain that $\lambda_{\e}^{(m)} \neq \lambda_{\e}^{(\ell)}$ when $\e>0$ is small enough. Thus, $\displaystyle \langle f^{(\ell)}_{\e}, D_\beta \overline{\displaystyle \displaystyle f^{(m)}_{\e}}\rangle =0$.
\end{proof}

\begin{lemma}\label{cor:bas}
    For every $1\leq \ell\neq m\leq N$,  $\lim_{\e \to 0} \langle \displaystyle f^{(\ell)}_{\e},f^{(m)}_{\e} \rangle =0$.
\end{lemma}
\begin{proof} 
Let $s_\ell,s_m\in\{1,\ldots,S\}$ such that $\ell\in B_{s_\ell}$ and $m\in B_{s_m}$.

Consider a converging subsequence of $\displaystyle\{\langle f_{\e_n}^{(\ell)},f_{\e_n}^{(m)}\rangle \}_{n\in \mathbb N}\subset \{\langle f_{\e}^{(\ell)},f_{\e}^{(m)}\rangle\}_{\e>0}$. 
By passing through a subsequence, if necessary, there exist complex vectors $\displaystyle f_*^{(\ell)},f_*^{(m)}\in \C^N$
$$\displaystyle \lim_{n\to\infty}f_{\e_n}^{(\ell)} = f_*^{(\ell)} \ \text{and }\displaystyle \lim_{n\to\infty}f_{\e_n}^{(m)} = f_*^{(m)}. $$

From Lemma \ref{lem:46}, we have that $\displaystyle f_*^{(m)}$ and $\displaystyle f_*^{(\ell)}$ are eigenvectors of the matrix $\hat{P}_{\beta}$ (see \ref{eq:Tb}). Moreover, Lemma \ref{lem:zero} implies that $\displaystyle f_*^{(\ell)}\in V_{s_\ell}$ and $\displaystyle f_*^{(\ell)}\in V_{s_m}$. Therefore, $f_*^{(\ell)}$ and $f_*^{(m)}$ are eigenvectors of the real symmetric matrix $\hat{W}_{L}$. Hence, $\overline{\displaystyle f_*^{(m)}} = a f_*^{(m)}$ for some $a\in\mathbb C.$ Finally, using Lemma \ref{lem:ld} we obtain that for every $m\neq\ell$
\begin{align*}
0 &= \lim_{n\to\infty}  \frac{e^{-2\pi i k \beta_{s_m}}}{\overline{a}} \left\langle f_{\e_n}^{(\ell)}, \displaystyle D_\beta \overline{\displaystyle f_{\e_n}^{(m)}} \right\rangle =\frac{e^{-2\pi i k \beta_{s_m}}}{\overline{a}} \left\langle f_*^{(\ell)}, D_\beta \overline{\displaystyle f_*^{(m)}} \right\rangle \\
 &=\langle\displaystyle f_*^{(\ell)}, f_*^{(m)}\rangle = \displaystyle\lim_{n\to\infty} \langle f_{\e_n}^{(\ell)},  f_{\e_n}^{(m)} \rangle
\end{align*}
which finishes the proof of the lemma.

\end{proof}
With the above results established, we can finally prove Theorem \ref{thm:3}.

\begin{proof}[Proof of Theorem \ref{thm:3}]

We start by establishing \eqref{eq:lambda1staprox}, which implies the first part of $(ii)$. Let $\ell\in\{1,\ldots, N\}$, from Lemma \ref{prop:xama} then there exists $s\in\{1,\ldots,S\}$, such that $\lambda_{\e}^{(\ell)} \to e^{-2\pi i k \beta_s}.$ Define $\pi_s: \C^N \to V_s$ as the $\mathbb C^N$-orthogonal projection onto $V_s$ (cf. \eqref{eq:Vi}). 
Observe that
\begin{align}
\hat{P}_\beta f^{(\ell)}_{\e}  &= D_\beta\hat{W}_{L}\displaystyle f^{(\ell)}_{\e} = D_\beta\hat{W}_{L} \pi_s f_\e^{(\ell)}  + D_\beta\hat{W}_{L}( \displaystyle f^{(\ell)}_{\e} - \pi_s f_\e^{(\ell)} )\nonumber\\
&= \frac{1}{\e}\left(\lambda_{\e}^{(\ell)} - e^{-2 \pi i k \beta_s} \right) \pi_s f_\e^{(\ell)}   + \mathcal O\left(\max_{j\in \{1,\ldots,N\}\setminus B_s} |\displaystyle f^{(\ell)}_{\e}(j)| \right),\label{eq:lim}
\end{align}
\rev{where, for a given sequence $a_\e\to0$, the notation $\mathcal O(a_\e)$ denotes a quantity such that $\limsup_{\e\to0}|\mathcal O(a_\e)|/a_\e<\infty.$}

Assume for a contradiction that the sequence 
\begin{align}
    \left\{\frac{1}{\e}\left(\lambda_{\e}^{(\ell)} - e^{-2 \pi i k \beta_s} \right)\right\}_{\e>0}\label{eq:le}
\end{align} does not converge as $\e\to 0$. From Lemma \ref{prop:xama}, \eqref{eq:le} is bounded for $\e>0$ small enough. Since \eqref{eq:le} is bounded and we are assuming that the limit of such a sequence does not exist, there exist  sequences of real numbers $\{\e_{n}\}_{n\in\mathbb N}$ and $\{\e_n'\}_{n\in\mathbb N}$ and complex numbers  $ \lambda_{(1)}\neq \lambda_{(2)}\in\mathbb C$  such that:
\begin{enumerate}
    \item $\e_{n},\e'_{n}\xrightarrow[]{n\to \infty} 0$;
    \item $\e_{n} \leq  \e_{n}' \leq  \e_{n+1} \ \text{for every }n\in\mathbb N$; and
    \item $\frac{1}{\e_n}\left(\lambda_{\e_n}^{(\ell)} - e^{-2 \pi i k \beta_s} \right)\xrightarrow[]{n\to \infty} \lambda_{(1)}$ and $\frac{1}{\e_n'}\left(\lambda_{\e_n'}^\ell - e^{-2 \pi i k \beta_s} \right)\to \lambda_{(2)}.$
\end{enumerate}

Using that $\e \in (0,\gamma_k) \mapsto \lambda_{\e}^{(\ell)}\in \mathbb C$ is continuous we can construct a sequence of real numbers $\{\e_r''\}_{r''\in\mathbb N}$ such that
\begin{enumerate}
    \item $\e_{n}\leq \e''_{n} \leq \e'_{n}$ for every $n\in \mathbb N$; 
    \item $\frac{1}{\e_n''}\left(\lambda_{\e_n''}^\ell - e^{-2 \pi i k \beta_s} \right)\xrightarrow[]{n\to \infty} \lambda_{(3)} \not\in \sigma(D_\beta \hat{W}_{L})$; and
    \item $f^{(\ell)}_{\e_n''} \to f^*\in V_s$ with $\|f^*\|_2 = 1$.
\end{enumerate}

Taking the limit $\e_r''\to 0$ in \eqref{eq:lim} we obtain that 
\begin{align*}
    D_\beta\hat{W}_{L}f^* &= \hat{P}_{\beta} f^* = \lim_{\e_r''\to 0} \hat{P}_\beta f^{(\ell)}_{\e_r''} \\
    &= \lim_{\e_r''\to 0}\left[\frac{1}{\e_r''}\left(\lambda_{\e_r''}^{(\ell)} - e^{-2 \pi i k \beta_s} \right) \pi_s f_{\e_r''}^{(\ell)}   + \mathcal O\left(\max_{j\in \{1,\ldots,N\}\setminus B_s} |\displaystyle f^{(\ell)}_{\e_r''}(j)|\right)\right]\\
    &=   \lambda_{(3)} f^* + 0 =  \lambda_{(3)} f^*.
\end{align*} 
Since $\lambda_{(3)} \not \in \sigma (D_\beta \hat{W}_{L})$ this yields a contradiction, therefore
$\left\{\frac{1}{\e}\left(\lambda_{\e}^{(\ell)} - e^{-2 \pi i k \beta_s} \right)\right\}_{\e>0}$
must converge to an eigenvalue $\hat{\lambda}^{(\ell)}$ of $D_\beta \hat{W}_{L}.$ We thus conclude that $\eqref{eq:lambda1staprox}$ holds.

In the following we show $(i)$. Recall that from the assumption imposed on $\hat{W}_L$ (see Definition \ref{Ladmiss} (3)) that $\hat{P}_\beta    = D_\beta \hat{W}_L$ admits $L_s=\#B_s$ distinct eigenvalues lying in $V_s$ and that $\ell \in B_s$. From \eqref{eq:lim} and Lemma \ref{lem:zero} it is clear that all the accumulation points of $\displaystyle\{\displaystyle f^{(\ell)}_{\e}\}_{\e>0}$ lie in the one-dimensional complex vector space $\mathrm{ker}(  \hat{P}_{\beta}  - \hat{\lambda}^{(\ell)})\cap V_s$. Choosing $f^{(\ell)} \in V_s\setminus \{0\}$ such that $\hat{P}_{\beta}f^{(\ell)} = \lambda^{(\ell)} f^{(\ell)}$ we have that  $[\displaystyle f^{(\ell)}_{\e}] \xrightarrow[]{\e \to 0} [f^{(\ell)}]$ in $\mathbb C \mathbb P^{N-1}$. Since  $\hat{P}_{\beta}$ has $\#B_s$ different eigenvectors lying in $V_s$, using Lemma \ref{cor:bas} we obtain that $\{\displaystyle f^{(\ell')}\}_{\ell'\in B_s}$ is an orthonormal basis of $V_s$. Observing that $\C^{N} = V_1 \oplus \ldots \oplus V_S$ and repeating the same argumentation presented in this paragraph for every for every $s'\in \{1,\ldots, S\}$ we obtain that $\{f^{(1)},\ldots, f^{(N)} \}$ is an orthonormal basis of $\mathbb C^{N}$, which concludes the proof of $(i)$.

Finally, we show the second part of $(ii)$, i.e. we show that $\mathrm{arg}(\hat{\lambda}^{(\ell)}) = \mathrm{arg}(\lambda^{(\ell)}) +\pi$ for every $\ell \in\{1,\ldots,N\}$ if $\lambda^{(\ell)}\neq 0$. 
Since $\hat{P}_{\beta}=D_{\beta} \hat{W}_{L}$ is a block matrix, there exists $s\in\{1,\ldots, S\}$ such that $\lambda^{(\ell)}$ is also an eigenvalue of the $L_s \times L_s$ matrix $e^{-2\pi i k \beta_s} \hat{W}_{s}$. In this way, $\lambda^{(\ell)} = e^{-2\pi i k \beta_s} \rho_s$ where $\rho_s\in \R$ is an eigenvalue of the real symmetric matrix $\hat{W}_s$. From Definition \ref{Ladmiss} (1), all the diagonal entries of $\hat{W}_s$ are negative real numbers and the sum of each row is non-positive. As a consequence of the Gershgorin circle theorem \cite[Theorem 6.11]{MAnalysis}, we obtain that all eigenvalues of $\hat{W}_s$ are non-positive, i.e. $\rho_s\leq 0$. Hence, the complex argument of $\hat{\lambda}^{(\ell)} = e^{-2\pi i k \beta_s} \rho_s$ is equal to $-2\pi k \beta_s + \pi\ (\text{mod }2\pi)$, as long as $\hat{\lambda}^{(\ell)}\neq 0.$  
\end{proof}


The following theorem discusses the linear response of the eigenprojections $\Pi_{k,\e}^{(\ell)}$.

\begin{theorem}\label{thm:fhat}

Let $\ell\in\{1,\ldots,N\}$, $s_\ell\in\{1,\ldots, S\}$ such that $\ell \in B_{s_\ell}$ and $\beta \in \Gamma$. For each $\e>0$ small enough, consider $\{\displaystyle f^{(1)}_{\e},\ldots, f^{(N)}_{\e}\}$ and $\{f^{(1)},\ldots, f^{(N)}\}$ as in Theorem \ref{thm:3}. Then, 

\begin{enumerate}
    \item[$(i)$] there exist vectors $\hat{f}^{(1)},\ldots,\hat f^{(N)}\in\mathbb{C}^N,$ so that $$\left[f_\e^{(\ell)}\right] = \left[f^{(\ell)} + \e \hat{f}^{(\ell)} + \smallO(\e)\right]\ \text{for every }\ell \in\{1,\ldots,N\},$$ 
where
\begin{align}
 \hat{f}^{(\ell)}=& \sum_{s=1, s\neq s_\ell}^{S} \left(\sum_{r \in B_{s_\ell}\setminus\{ \ell\} }    \frac{1}{\hat\lambda^{(\ell)} - \hat\lambda^{(r)}}\frac{1}{ e^{-2\pi i k \beta_{s} }- e^{-2\pi i k \beta_{s_\ell}} } \left\langle 
 \pi_s D_\beta \dot{W} f^{(\ell)} ,\pi_s  \dot{W} D_{-\beta} f^{(r)}  \right\rangle f^{(r)} \right.\nonumber\\
& +  
\left.\sum_{r \in B_s}\frac{1}{e^{-2\pi i k \beta_s} - e^{-2\pi i k \beta_{s_\ell} }} \left\langle  D_\beta \dot{W} f^{(\ell)} , f^{(r)}   \right\rangle f^{(r)}\right) .\label{eq:fhatl}
\end{align}
and  $\pi_s:\C^N\to V_s$ is the $\mathbb C^N$-orthogonal projection onto $V_s$.
\item[$(ii)$]Moreover,
\begin{align*}
    \Pi_{\e}^{(\ell)} =  \Pi^{(\ell)} + \e   \left( \left\langle \cdot, F^{(\ell)}\right\rangle \hat{F}^{(\ell)} +  \left\langle \cdot, \hat{F}^{(\ell)} \right\rangle F^{(\ell)} \right) + \smallO(\e),
\end{align*}
where $F^{(\ell)}(j,x) = f^{(\ell)}(j)e^{2\pi i k x}$ and $\hat{F}^{\ell}(j,x):= \hat{f}^{\ell}(j)e^{2\pi i k x}.$
\item[$(iii)$] Finally, $\langle F_k^{(\ell)}, \hat{F}_k^{(\ell)}\rangle = 0$ and the term $\smallO(\e):L^2(M)\to L^2(M)$ in $\eqref{evecresponse}$ is a finite-rank operator such that $\|\smallO(\e)\|/\e \xrightarrow[]{\e\to 0}0.$
\end{enumerate}

\end{theorem}


\subsection{Proof of Theorem \ref{thm:fhat}}
To prove the above theorem, we first establish four preliminary results.

\begin{proposition}\label{prop:nicefam}

Let $\{f^{(1)},\ldots, f^{(N)}\} \subset \mathbb{C}^N$ be the orthogonal eigenbasis from Theorem~\ref{thm:3}. Then, there exists $\gamma > 0$ and a family  
$\{ f_\e^{(1)},\ldots,f_\e^{(N)} \}_{0<\e<\gamma} \subset \mathbb{C}^N$  
of unit-norm eigenbases of $P_{\e,\beta}$ such that, for every $\ell \in \{1,\ldots, N\}$
\begin{enumerate}
    \item[(i)] the map $\e \in (0,\gamma)\mapsto\displaystyle  f_\e^{(\ell)}\in \C^N$ is $\mathcal C^{\infty}$; and
    \item[(ii)] $\displaystyle f_\e^{(\ell)}\xrightarrow[]{\e\to 0} f^{(\ell)}$ in $\mathbb C^N$.
\end{enumerate}
\end{proposition}
\begin{proof}
Fix a natural number $\ell \in \{1,\ldots, N\}$. As established at the beginning of Section \ref{sec:41T}, by means of perturbation theory, there exists a family $\{\displaystyle g_\e^{(1)},\ldots, g_\e^{(N)}\}_{{0<\e<\gamma}}\subset \C^N$ of unit-norm eigenbasis of $\mathcal P_{\e,\beta}$ such that $(i)$ is satisfied.

From Theorem \ref{thm:3}, it follows that $[\displaystyle g_\e^{(\ell)}] \to [f^{(\ell)}]$ as $\e\to 0$. Since $\|f^{(\ell)}\|=1$ and $\|g_\e^{(\ell)}\|=1$, we obtain that  $|\langle f^{(\ell)} , g_\e^{(\ell)} \rangle | \to 1$ as $\e\to 0.$
Shrinking $\gamma>0$, if necessary, we can assume that $1/2 < \|\langle f^{(\ell)} , g_\e \rangle \|$ for $0<\e<\gamma$.

Define
$$f_{\e}^{(\ell)} := \frac{\langle f^{(\ell)} , g_\e \rangle}{|\langle f^{(\ell)} , g_\e \rangle|} g_\e\ \text{for every }0<\e<\gamma.$$ 
From a direct computation, we obtain that
$$\lim_{\e\to 0 } \langle f_\e^{(\ell)}, f^{(\ell)} \rangle = \lim_{\e\to 0 } \left\langle \frac{\langle f^{(\ell)} , g_\e \rangle}{|\langle f^{(\ell)} , g_\e \rangle|} g_\e ,f^{(\ell)}\right\rangle = \lim_{\e\to 0 } \frac{\langle f^{(\ell)} , g_\e \rangle \left\langle f_\e , g_\e \right\rangle}{|\langle f^{(\ell)} , g_\e \rangle|} = \lim_{\e\to 0 } |\langle g_
\e, f^{(\ell)} \rangle| =1.$$
This implies that $\displaystyle f_\e^{(\ell)} \to f^{(\ell)}$ as $\e \to 0$, since $\displaystyle f_\e^{(\ell)}$ is a unit-norm vector for every $0<\e<\gamma$.



\end{proof}


In what follows, we establish some notation. We fix the orthonormal basis 
$\{f^{(1)},\ldots,f^{(N)}\}$ $\subset \C^N$ given in Theorem \ref{thm:3} and a family 
$\{\displaystyle f_\e^{(1)},\ldots,f_\e^{(N)}\}_{0<\e<\gamma}\in \C^N$ of  unit-norm eigenbases of $P_{\e,\beta}$ provided by Proposition~\ref{prop:nicefam}. In particular, we obtain that $f_
\e \xrightarrow[]{\e\to 0} f^{(\ell)}$ for every $\ell\in \{1,\ldots, N\}$.

For each \( \ell \in \{1, \dots, N\} \) and sufficiently small \( \e \), we define \( \mu^{(\ell)}_\e \) as the eigenvectors of the Hermitian conjugate $(P_{\e,\beta})^*$, which satisfy $(P_{\e,\beta})^* \mu^{(\ell)}_\e = \overline{\lambda^{(\ell)}_\e} \mu^{(\ell)}_\e.$ Furthermore, we impose the normalization condition  
\begin{align}
    \langle f_\mu^{(\ell)}, \mu^{(\ell)}_\e \rangle = 1 \ \text{for all}\ \ell \in \{1, \dots, N\}\ \text{and}\ \e > 0.\label{eq:mu1}
\end{align}

By Theorem \ref{thm:gamma} for $\e>0$ small enough, $P_{\e,\beta}$ admits $N$ distinct (and therefore simple) eigenvalues, and by appealing to standard matrix perturbation theory, the derivatives \({\displaystyle\dot{f}^{(\ell)}_{\e}}  = \frac{\partial}{\partial \e} \displaystyle f^{(\ell)}_{\e}\) exist for $\e>0$. The strategy of the proof of Theorem \ref{thm:fhat} is to understand the behaviour of  $\displaystyle \dot{f}^{(\ell)}_{\e}\ \text{as}\ \e\to 0.$ To achieve this, for every $\e>0$ small enough we write $\dot{f}_\e^{(\ell)}$ in the basis
$$\{f^{(1)}_{\e},\ldots,f_{\e}^{(\ell-1)}, \mu_\e^{(\ell)},f_{\e}^{(\ell+1)} \ldots, f^{(N)}_{\e}\},$$
i.e.
\begin{align}
    \displaystyle\df{\e}{\ell} = \alpha_{\e}^\ell \mu_\e^{(\ell)}+ \sum_{m=1,i\neq \ell}^{N} \alpha_\e^{m} \displaystyle \f{\e}{m},\label{eq:alpha}
\end{align}
and we study the behaviour of the complex numbers $\alpha_\e^{(1)},\ldots,\alpha_\e^{(N)}$ as $\e\to 0.$

Observe that by repeating the proof of Theorem \ref{thm:3}, but considering the matrix $(P_{\e,\beta})^* = W_\e D_{-\beta}$ instead of $P_{\e,\beta}$, and noting that the matrices $\hat{W}_{L} D_{-\beta}$ and $D_\beta \hat{W}_{L}$ have the same eigenvectors, we obtain that 
\begin{align}
   \displaystyle\mu^{(\ell)}_\e \xrightarrow[]{\e \to 0} f^{(\ell)}.\label{eq:mu3} 
\end{align} Moreover, it is worth noting that  
\begin{align}
\langle f^{(m)}_{\e},\mu_\e^{(\ell)} \rangle =0 \quad \text{if } m\neq \ell.\label{eq:mu2}    
\end{align}

Indeed,
$$\langle f^{(m)}_{\e}, \mu_\e^{(\ell)} \rangle = \frac{1}{\lambda^{(m)}_\e} \langle P_{\e,\beta} f^{(m)}_{\e}, \mu_\e^{(\ell)} \rangle =  \frac{1}{\lambda^{(m)}_{\e}} \langle  f^{(m)}_{\e} , (P_{\e,\beta})^* \mu_\e^{(\ell)} \rangle  = \frac{\lambda^{(\ell)}_{\e}}{\lambda^{(m)}_{\e}} \langle f_\e^{(m)}, \mu_\e^{(\ell)} \rangle.$$
Since $\lambda_\e^{(\ell)} \neq\lambda_{\e}^{(m)}$, it follows that $\displaystyle\langle f^{(m)}_{\e},\mu_\e^{(\ell)} \rangle=0$ for $\ell \neq m$.


With the above notation and initial observations established, we proceed to prove four preliminary results, which will culminate in Theorem \ref{thm:fhat}.


\begin{proposition}\label{prop:linresponselast}

Let $\ell\in\{1,\ldots,N\}$ and $s_\ell\in\{1,\ldots,S\}$ such that $\ell \in B_{s_\ell}$, then
\begin{align}
    \lim_{\e\to 0} \frac{ \left(\displaystyle f^{(\ell)}_{\e}\right)_j - \left(f^{(\ell)}\right)_j}{\e } = \lim_{\e\to 0} \frac{ \left(\displaystyle f^{(\ell)}_{\e}\right)_j}{\e}= \frac{1}{\ex{k\beta_{s_\ell}}-\ex{k\beta_{s_j}}} \left(D_\beta \dot{W} f^{(\ell)}\right)_j\label{fresp}
\end{align}
 for each $j\in\{1,\ldots,N\}\setminus B_{s_\ell},$ where $s_j \in \{1,\ldots, S\}$ satisfies $j\in B_{s_j}.$
\end{proposition}
\begin{proof}
Consider $j\in\{1,\ldots,N\}\setminus B_{s_\ell}$, and $s_j\in\{1,\dots,S\}$ such that $j\in B_{s_j}.$ Recall from Lemma \ref{lem:zero} that $(f^{(\ell)})_j =0.$  On the one hand, we obtain that
\begin{align*}\frac{1}{\e}\left(P_{\e,\beta} {\displaystyle f^{(\ell)}_{\e}} - D_\beta {f^{(\ell)}}\right)_j &=\frac{1}{\e} \left(D_\beta(W_\e -\mathrm{Id} ) {\displaystyle f^{(\ell)}_{\e}}+D_\beta(\displaystyle f^{(\ell)}_{\e} - {f^{(\ell)}})\right)_j\\
&= \left(D_\beta\left(\frac{W_\e -\mathrm{Id}}{\e} \right) {\displaystyle f^{(\ell)}_{\e}}\right)_j + \frac{e^{2\pi i k \beta_{s_j}}}{\e} \left(\displaystyle f^{(\ell)}_{\e} - {f^{(\ell)}}\right)_j
\end{align*}
On the other hand, since $f^{(\ell)}\in V_{s_\ell}$ (see \eqref{eq:Vi}),
\begin{align*}
    \frac{1}{\e}\left(P_{\e,\beta} {\displaystyle f^{(\ell)}_{\e}} - D_\beta {f^{(\ell)}}\right)_j &=  \frac{1}{\e}\left(\lambda_{\e}^{(\ell)} \displaystyle f^{(\ell)}_{\e} - e^{-2\pi i k \beta_{s_\ell}} f_\e^{(\ell)} + e^{-2\pi i k \beta_{s_\ell}} f_\e^{(\ell)}- e^{-2\pi i k \beta_{s_\ell}} {f^{(\ell)}} \right)_j\\ &=\frac{1}{\e}\left(\lambda_{\e}^{(\ell)} - e^{-2\pi i k \beta_{s_\ell}}\right)\displaystyle \left(f^{(\ell)}_{\e}\right)_j + \frac{e^{-2\pi i k \beta_{s_\ell}}}{\e} \left( \displaystyle f^{(\ell)}_{\e} -  {f^{(\ell)}} \right)_j.
\end{align*}
 
 Combining the above equations we obtain that
$$  \left(e^{-2\pi i k\beta_{s_j}}- e^{-2\pi i k\beta_{s_\ell}}\right)\left(\frac{\displaystyle (f_{\e}^{(\ell)})_j - {(f^{(\ell)}})_j }{\e}\right)= \frac{\lambda_{\e}^{(\ell)} - e^{-2\pi i k \beta_{s_\ell}}}{\e} \displaystyle \left(\f{\e}{\ell}\right)_j -\left( D_\beta\left(\frac{W_\e -\mathrm{Id}}{\e} \right) {\displaystyle f^{(\ell)}_{\e}}\right)_j.$$
From Lemma \ref{lem:zero} it follows that $\displaystyle  \left(\displaystyle f_{\e}^{(\ell)}\right)_j\xrightarrow[]{\e\to 0} 0$ for every $j\in\{1,\ldots,N\}\setminus B_{s_\ell}.$ Hence,
\begin{align*}
    \lim_{\e\to 0} \frac{\left(\displaystyle f^{(\ell)}_{\e}\right)_j - \left({f^{(\ell)}}\right)_j}{\e } &= \frac{1}{\ex{k\beta_{s_j}}-\ex{k\beta_{s_\ell} }}\left(\hat{\lambda}^{(\ell)}{\left(f^{(\ell)}\right)}_j - \left(D_\beta \dot{W} f^{(\ell)}\right)_j\right)\nonumber \\&= 
    \frac{1}{\ex{k\beta_{s_\ell}}-\ex{k\beta_{s_j}}} \left(D_\beta \dot{W} f^{(\ell)}\right)_j.
\end{align*}
\end{proof}

\begin{remark}\label{rmk:resp}
\rev{Proposition~4.7 gives an explicit componentwise formula for $(\hat f^{(\ell)})_j$ at indices $j\notin B_{s_\ell}$, where $\hat f^{(\ell)}$ is the response vector introduced in Theorem \ref{thm:fhat} (proved at the end of this section). In particular, \eqref{fresp} shows that outside the band $B_{s_\ell}$ the response is completely determined by the corresponding entries of $D_\beta\dot W f^{(\ell)}$, scaled by the factor $\left(\ex{k\beta_{s_\ell}}-\ex{k\beta_{s_j}}\right)^{-1}$. If $\dot W$ is supported close to the diagonal (for example, if $\dot W$ is tridiagonal), then applying $\dot W$ to a vector supported in $B_{s_\ell}$ can only produce nonzero entries in entries close to $B_{s_\ell}$ (with width controlled by the bandwidth). Consequently, $\left(\hat f^{(\ell)}\right)_j=0$ for all indices $j$ outside that entry neighbourhood.}
\end{remark}

\begin{proposition} \label{prop:Ae}
  Under the above notation of equation \eqref{eq:alpha}, the vector ${\boldsymbol\alpha}_\e = (\alpha_\e^1,\ldots, \alpha_\e^N)$ satisfies ${\boldsymbol\alpha}_\e =\left(A_{\beta}^\e\right)^{-1} \mathbf{v}_{\beta}^{\e}$
where $A_{\beta}^\e$ is a $N\times N$ matrix such that
$$ [A_{\beta}^\e]_{r,s} = \begin{cases}
    1, & \text{if }s=\ell\\
    \langle \mu_\e^{(\ell)}, \mu_\e^{(r)}    \rangle ,& \text{if }r \neq \ell\ \text{and}\ s = \ell\\
 \langle \displaystyle \f{\e}{s}, \displaystyle \f{\e}{\ell}     \rangle   ,&\text{if }\ell=r\ \text{and }s\neq \ell\\
    0,&\text{otherwise}.
\end{cases} $$
and
$$\mathbf{v}_{\beta}^{\e} =\left(\langle \displaystyle\df{\e}{\ell}, \mu_\e^{(1)} \rangle, \ldots, \langle \displaystyle\df{\e}{\ell}, \mu_\e^{(\ell-1)}  \rangle,\langle \displaystyle\df{\e}{\ell},  \displaystyle \f{\e}{\ell}\rangle,\langle \displaystyle\df{\e}{\ell}, \mu_\e^{(\ell+1)}    \rangle ,\ldots, \langle \displaystyle\df{\e}{\ell}, \mu_\e^{(N)}  \rangle  \right).$$ 
\end{proposition}
\begin{proof}
    Observe that for every $r\in \{1,\ldots,\ell-1,\ell+1,\ldots,N\}$ we obtain that
$$ \langle \displaystyle\df{\e}{\ell}, \mu_\e^{(r)}\rangle = \left\langle\alpha_{\e}^\ell \mu_\e^{(\ell)}+ \sum_{m=1,m\neq r}^{N} \alpha_\e^{m} \displaystyle \f{\e}{m}, \mu_\e^{(r)}\right\rangle  = \alpha_\e^{r} + \alpha_{\e}^\ell \langle \mu_\e^{(\ell)},\mu_\e^{(r)} \rangle$$
and
$$\langle \displaystyle\df{\e}{\ell}, \mu_\e^{(\ell)}\rangle = \alpha_\e^\ell + \sum_{m=1,m\neq \ell}^{N} \alpha_{\e}^m\langle  \displaystyle \f{\e}{m}, \displaystyle \f{\e}{\ell} \rangle.$$
By solving the above linear equation, we obtain the desired result.
\end{proof}

\begin{proposition}\label{prop:linresponselastmu} Let $\ell\in\{1,\ldots,N\}$ and $s_\ell \in\{1,\ldots,S\}$ such that $N_{s_\ell-1}< \ell\leq N_{s_\ell}$. Then, 
$$\lim_{\e\to 0} \frac{\displaystyle \left(\mu^{(\ell)}_\e\right)_j - \left(\mu^{(\ell)}_\e\right)_j}{\e } = \lim_{\e\to 0} \frac{\displaystyle \left(\mu^{(\ell)}_\e\right)_j }{\e } = \frac{1}{e^{2\pi i k\beta_{s_\ell}}-e^{2\pi i k\beta_{s_j}}} \left(\dot{W} D_{-\beta} f^{(\ell)}\right)_j$$
for each $j\in\{1,\ldots, N\}\setminus B_{s_\ell}$, and  $s_j\in\{1,\ldots,N\}$ is such that $j\in B_{s_j}$.

\end{proposition}
\begin{proof}
  The proof follows by repeating the computations in presented in Proposition \ref{prop:linresponselast} and recalling from \eqref{eq:mu3} that $\mu_\e^{(\ell)}\xrightarrow[]{\e\to 0} f^{(\ell)}$ for every $\ell\in\{1,\ldots,N\}.$  
\end{proof}



\begin{proposition}\label{pro:InnerProdfe}
Let $\ell \in \{1,\ldots, N\}$ and $s_\ell \in \{1,\ldots,S\}$ such that $\ell \in B_{s_\ell}$. Then, given $r \in \{1,\ldots, N\}$ and $s_r\in\{1,\ldots,S\}$ such that $r \in B_{s_r},$
\begin{enumerate}
    \item[$(a)$]  if  $s_r \neq s_\ell,$ 
   \begin{align*}
\lim_{\e\to 0}\langle\dot{f}^{(\ell)}_{\e},\mu_\e^{(r)}\rangle 
=&\frac{1}{e^{-2\pi i k \beta_{s_r} } - e^{-2\pi i k \beta_{s_\ell}}} \left\langle  D_\beta \dot{W} f^{(\ell)} , f^{(r)}\right\rangle;
\end{align*}
    \item[$(b)$] if $s_r = s_\ell$ and $r\neq\ell$ 
    \begin{align*}
\lim_{\e\to 0}\langle\dot{f}^{(\ell)}_{\e},\mu_\e^{(r)}\rangle=&\frac{1}{\hat\lambda^{(\ell)} - \hat\lambda^{(r)}} \sum_{s=1, s\neq s_\ell}^{S} \frac{1}{e^{-2\pi i k \beta_{s}} - e^{-2\pi i k \beta_{s_\ell} }} \left\langle 
 \pi_s D_\beta \dot{W} f^{(\ell)} ,\pi_s  \dot{W} D_{-\beta} f^{(r)}  \right\rangle ;
\end{align*}
\noindent where $\pi_s:\C^N\to V_s$ is the $\mathbb C^N$-orthogonal projection onto $V_s$.
\end{enumerate}
\end{proposition}

\begin{proof} 
We start by showing $(a)$. Differentiating the equation $\displaystyle P_{\e,\beta} \displaystyle f^{(\ell)}_{\e} = \lambda_{\e}^{(\ell)} \displaystyle f^{(\ell)}_{\e}$ with respect to $\e$, and applying $\langle \cdot, \mu_\e^{(r)}\rangle$ in both sides we obtain from \eqref{eq:mu2} that
$$ \langle \dot{P}_{\e,\beta} \displaystyle f^{(\ell)}_{\e} ,\mu_\e^{(r)} \rangle   + \langle   P_{\e,\beta} \dot{f}^{(\ell)}_{\e},\mu_\e^{(r)}\rangle  = \dot{\lambda}_{\e}^{(\ell)}  \langle \displaystyle f^{(\ell)}_{\e},  \mu_\e^{(r)} \rangle+ \lambda_{\e}^{(\ell)} \langle \dot{f}^{(\ell)}_{\e} , \mu_\e^{(r)} \rangle.  $$
Hence $ \displaystyle \langle \dot{P}_{\e,\beta} \displaystyle f^{(\ell)}_{\e} ,\mu_\e^{(r)} \rangle  + \lambda_{\e}^{(r)} \langle    \dot{f}^{(\ell)}_{\e},\mu_\e^{(r)}\rangle =  \lambda_{\e}^{(\ell)} \langle    \dot{f}^{(\ell)}_{\e},\mu_\e^{(r)}\rangle,$ which implies
\begin{align}
    \langle    \dot{f}^{(\ell)}_{\e},\mu_\e^{(r)}\rangle=  \frac{1}{\lambda_{\e}^{(\ell)} - \lambda_{\e}^{(r)} }\langle \dot{P}_{\e,\beta} \displaystyle f^{(\ell)}_{\e} ,\mu_\e^{(r)} \rangle. \label{eq:inner1} 
\end{align}

Observe that $P_{\e,\beta} = D_\beta W_\e = D_\beta \left(\mathrm{Id} + \e \dot{W}\right)$, in this way
\begin{align}
    \dot{P}_{\e,\beta} \displaystyle f^{(\ell)}_{\e} = D_\beta \dot{W} \displaystyle f^{(\ell)}_{\e} = \frac{1}{\e}\left(\lambda_{\e}^{(\ell)} \displaystyle f^{(\ell)}_{\e} - D_\beta \displaystyle f^{(\ell)}_{\e}\right)\label{eq:Ff1}.
\end{align} 
Combining equations \eqref{eq:inner1} and \eqref{eq:Ff1} 
\begin{align}
    \langle\dot{f}^{(\ell)}_{\e},\mu_\e^{(r)}\rangle&=   \frac{1}{\lambda_{\e}^{(\ell)} - \lambda_{\e}^{(r)}}\left\langle \frac{1}{\e}\left(\lambda_{\e}^{(\ell)} \displaystyle f^{(\ell)}_{\e} - D_\beta \displaystyle f^{(\ell)}_{\e}\right) ,\mu_\e^{(r)} \right\rangle \nonumber \\
    &= \frac{1}{\e(\lambda_{\e}^{(\ell)} - \lambda_{\e}^{(r)}) }\langle D_\beta \displaystyle f^{(\ell)}_{\e} , \mu_\e^{(r)} \rangle \nonumber \\
    &= \frac{1}{\e(\lambda_{\e}^{(\ell)} - \lambda_{\e}^{(r)}) }\langle e^{-2\pi i k \beta_{s_\ell}} \displaystyle f^{(\ell)}_{\e} , \mu_\e^{(r)} \rangle +  \frac{1}{\e(\lambda_{\e}^{(\ell)} - \lambda_{\e}^{(r)}) } \langle (D_\beta-e^{-2\pi i k \beta_{s_\ell}})  (1-\pi_{s_\ell})f^{(\ell)}, \mu_\e^{(r)} \rangle\nonumber \\
    &= \frac{1}{\lambda_{\e}^{(\ell)} - \lambda_{\e}^{(r)}} \left\langle   \frac{1}{\e} \left(D_\beta -e^{-2\pi i k \beta_{s_\ell}}\right) (1-\pi_{s_\ell}) \displaystyle f_\e^{(\ell)}, \mu_\e^{(r)} \right\rangle \label{eq:key}
\end{align}
 where $\pi_{s_\ell}:\mathbb C^N \to V_{s_\ell}$ denotes the $\mathbb C^N$-orthogonal projection onto $V_s$. From Proposition \ref{prop:linresponselast} and \eqref{eq:mu3} it follows that if $r \in \{1,\ldots,N\}\setminus B_{s_\ell}$
\begin{align*}
\lim_{\e\to 0}\langle\dot{f}^{(\ell)}_{\e},\mu_\e^{(r)}\rangle =& \lim_{\e \to 0}\frac{1}{\lambda_{\e}^{(s_\ell)} - \lambda_{\e}^{(s_j)}} \left\langle   \left(D_\beta -e^{-2\pi i k \beta_{s_\ell}}\right)  \frac{(1-\pi_{s_\ell}) \displaystyle f_\e^{(\ell)}-0}{\e}, \mu_\e^{(r)} \right\rangle \\
=&\frac{-1}{e^{-2\pi i k \beta_{s_\ell}} - e^{-2\pi i k \beta_{s_r}}} \left\langle  D_\beta \dot{W} f^{(\ell)} , f^{(r)}   \right\rangle .
\end{align*}
implying $(a)$. 

We now show $(b)$. From \eqref{eq:mu1}, \eqref{eq:key} and Proposition \ref{prop:linresponselastmu} we obtain that if $s_r = s_\ell$ and $\ell\neq r$,
\begin{align*}
\lim_{\e\to 0}\langle\dot{f}^{(\ell)}_{\e},\mu_\e^{(r)}\rangle=& \lim_{\e \to 0} \left\langle   (D_\beta - e^{-2\pi i k \beta_{s_\ell}})\frac{(1-\pi_{s_\ell}) f_\e^{(\ell)}}{\e}, \frac{1}{\left({\overline{\hat\lambda^{(\ell)}}} - {\overline{\hat\lambda^{(r)}}}\right)\e + \smallO (\e)}\mu_\e^{(r)} \right\rangle\nonumber \\
=&\frac{-1}{\hat\lambda^{(\ell)} - \hat\lambda^{(r)}} \sum_{s=1, j\neq s_\ell}^{S} \frac{1}{e^{-2\pi i k \beta_{s_\ell}} - e^{-2\pi i k \beta_{s} }} \sum_{j = N_{s-1}+1}^{N_{s}} \left(D_\beta \dot{W} f^{(\ell)}\right)_j  \left(  \overline{\dot{W} D_{-\beta}f^{(r)} }\right)_j\\
=&\frac{-1}{\hat\lambda^{(\ell)} - \hat\lambda^{(r)}} \sum_{s=1, s\neq s_\ell}^{S} \frac{1}{e^{-2\pi i k \beta_{s_\ell}} - e^{-2\pi i k \beta_{s} }} \left\langle 
 \pi_s D_\beta \dot{W} f^{(\ell)} ,\pi_s  \dot{W} D_{-\beta} f^{(r)}  \right\rangle.
\end{align*}

\end{proof}

With the above results, we can finally prove Theorem \ref{thm:fhat}

\begin{proof}[Proof of Theorem \ref{thm:fhat}] 
We start by proving $(i)$. Recall from Definition \ref{def:proj} that $\pi: v\in \C^{N}\mapsto [v]\in\mathbb C \mathbb P^{N-1}$ denotes a $\mathcal C^\infty$ quotient map from $\mathbb C^N$ onto $\mathbb C \mathbb P^{N-1}$. Since $[v] = [tv]$ for every $t\in \mathbb C\setminus \{0\}$ and $v\in \mathbb C^N$, from the chain rule we obtain that 
\begin{align}
    \mathrm{D}_v \pi (s v) = 0\ \text{for every }s \in \mathbb C.\label{eq:d0}
\end{align}
Hence, from the inverse function theorem and 
\eqref{eq:feconv} we obtain that there exists $\delta>0$, such that for each $\e\geq 0$ small enough and every $\ell \in \{1,\ldots, N\}$ the map
$$v\in \left\{\displaystyle f^{(\ell)}_\e + w: w\in\mathbb C^{N}, \langle w,f_\e^{(\ell)}\rangle =0\ \text{and }\|w\|< \delta\right\}\mapsto [v] \in \mathbb C\mathbb P^{N-1}$$
is a diffeomorphism with its image, where $f_0^{(\ell)}$ should be interpreted as $f^{(\ell)}$.

Observe that $[\displaystyle f_\e^{(\ell)}] = [f^{(\ell)}+\e w +\smallO(\e)]$ for some $w\in \mathbb C^{N}$ satisfying $\langle w, f^{(\ell)}\rangle =0$ if and only if for every $\mathcal C^\infty$ function $g:\mathbb C \mathbb P^{N-1}\to \mathbb R$ it holds that
\begin{align}
    g\circ [f_\e^{(\ell)}] = g\circ [f^{(\ell)}] +  \e \mathrm{D}_{[f^{(\ell)}]} g\,   \mathrm{D}_{f^{(\ell)}} \pi\, w +\smallO(\e).   \label{eq:g}
\end{align}
In the following, we show that \eqref{eq:g} holds when taking $w = \hat{f}^{(\ell)}$ (see \eqref{eq:fhatl}).

For every $\e>0$ small enough and $\ell\in \{1,\ldots,S\}$ denote $\displaystyle \hat{f}_\e^{(\ell)}$ as the $\mathbb C^N$-orthogonal projection of $\displaystyle\dot{f}_\e^{(\ell)}$ (see \eqref{eq:alpha} and Proposition \ref{pro:InnerProdfe}) onto the linear space $\{v\in\mathbb C^{N}; \langle v,f_\e^{(\ell)}\rangle =0\}.$ From \eqref{eq:mu1} and \eqref{eq:mu2} we have that $$\mathrm{span}\left\{\mu_\e^{(1)},\ldots, \mu_\e^{(\ell-1)},\mu_\e^{(\ell+1)},\ldots,\mu_\e^{(N)}\right\} = \{v\in\mathbb C^{N}: \langle v,f_\e^{(\ell)}\rangle =0\}.$$
By combining Propositions~\ref{prop:Ae} and~\ref{pro:InnerProdfe} with the definition of $\hat{f}^{(\ell)}$ in \eqref{eq:fhatl}, we obtain that
\begin{align}
    \lim_{\e \to 0} \hat{f}^{(\ell)}_\e  =&   \sum_{s=1, s\neq s_\ell}^{S} \sum_{r \in B_{s_\ell}\setminus\{\ell\} }  \frac{1}{\hat\lambda^{(\ell)} - \hat\lambda^{(r)}}\frac{1}{ e^{-2\pi i k \beta_{s} }- e^{-2\pi i k \beta_{s_\ell}} } \left\langle 
 \pi_s D_\beta \dot{W} f^{(\ell)} ,\pi_s   \dot{W} D_{-\beta} f^{(r)}  \right\rangle f^{(r)} \nonumber\\
& +  \sum_{s=1, s\neq s_\ell}^{S} \sum_{r \in B_s}\frac{1}{e^{-2\pi i k \beta_s} - e^{-2\pi i k \beta_{s_\ell} }} \left\langle  D_\beta \dot{W} f^{(\ell)} , f^{(r)}   \right\rangle f^{(r)} =\hat{f}^{(\ell)}.\label{eq:s1}
\end{align}


From the mean value theorem we obtain that for every $\e>0$ small enough, there exists $c_\e \in(0,\e)$ such that
$$\displaystyle g\circ \left[\f{\e}{\ell}\right] - g\circ\left[ f^{(\ell)}\right] = \e \left.\frac{\d}{\d t} \left( g \circ \pi\circ f_t^{(\ell)}\right) \right|_{t=c_\e} =\displaystyle \e  \mathrm{D}_{[f_{c_\e}]}g\, \mathrm{D}_{f_{c_\e}^{(\ell)}} \pi\, \dot{f}^{(\ell)}_{c_\e}.$$
By the definition of $\displaystyle\hat{f}_{c_\e}^{(\ell)}$, we that $\langle \displaystyle \hat{f}_{c_\e}^{(\ell)},f_{c_\e}^{(\ell)} \rangle =0 $ and $\dot{f}_{c_\e}^{\ell} = \hat{f}_{c_\e}^{\ell} + b_\e f_{c_\e}$ for some $b_\e \in \mathbb C$. In view of equation \eqref{eq:d0} it follows that
\begin{align}
    \displaystyle \mathrm{D}_{f_{c_\e}^{(\ell)}} \pi\,  \dot{f}^{(\ell)}_{c_\e} = \mathrm{D}_{f_{c_\e}^{(\ell)}} \pi\,  (\hat{f}^{(\ell)}_{c_\e}+ b_\e f_\e^{(\ell)}) 
 =\mathrm{D}_{f_{c_\e}^{(\ell)}} \pi\, \displaystyle\hat{f}^{(\ell)}_{c_\e}.\label{eq:ga}
\end{align}
Hence, combining \eqref{eq:s1},  \eqref{eq:feconv} and \eqref{eq:ga} we obtain that
$$\lim_{\e\to 0}\frac{ \displaystyle g \circ \left[\f{\e}{\ell}\right] - g\circ \left[f^{(\ell)}\right]}{\e} = \mathrm{D}_{[f^{(\ell)}]}g\,  \mathrm{D}_{f^{(\ell)}}\pi\, \hat{f}^{(\ell)}.$$
Thus, $[f_\e^{\ell}] = [f^{(\ell)} +\e \hat{f}^{(\ell)} + \smallO(\e)]$ and by \eqref{eq:feconv}, \eqref{eq:s1} and the property $\displaystyle \langle f^{(\ell)}_\e , \hat{f}_\e^{(\ell)}\rangle =0$ we have $\langle f^{(\ell)}, \hat{f}^{(\ell)}\rangle =0$,  concluding the proof of $(i).$

In the following, we prove $(ii)$. Since $[f_\e^{\ell}] = [f^{(\ell)} +\e \hat{f}^{(\ell)} + \smallO(\e)]$, $\displaystyle F_\e^{(\ell)} = f_\e^{(\ell)} e^{2\pi i k x}$ and $\hat{F}^{(\ell)}(j,x) := \hat{f}^{(\ell)}(j) e^{2\pi i k x}$ for every $\e>0$ small enough, it holds that
 $$\Pi^{(\ell)}_\e(\cdot)  = \displaystyle \left\langle \cdot, F_\e^{(\ell)} \right\rangle   F_\e^{(\ell)} = \frac{ \displaystyle \left\langle \cdot, F^{(\ell)} + \e \hat{F}^{(\ell)} +\smallO(\e) \right\rangle}{\| F^{(\ell)} + \e \hat{F}^{(\ell)} +\smallO(\e)\|^2} \left( F^{(\ell)} + \e \hat{F}^{(\ell)} +\smallO(\e)\right).$$
Since $\langle F^{(\ell)} , \hat{F}^{(\ell)}\rangle = 0,$ we obtain that
$$\frac{1}{\| F^{(\ell)} + \e \hat{F}^{(\ell)} +\smallO(\e)\|^2} = \frac{1}{\langle F^{(\ell)} + \e \hat{F}^{(\ell)} +\smallO(\e),\langle F^{(\ell)} + \e \hat{F}^{(\ell)} +\smallO(\e)\rangle} =1+\smallO(\e).$$
In this way,
\begin{align}
  \Pi^{(\ell)}_\e(\cdot)&=  \displaystyle \left\langle \cdot, F^{(\ell)} + \e \hat{F}^{(\ell)} +\smallO(\e) \right\rangle \left( F^{(\ell)} + \e \hat{F}^{(\ell)} +\smallO(\e)\right)\nonumber\\
  &=  \Pi^{(\ell)}(\cdot) + \e   \left( \left\langle \cdot, F^{(\ell)}\right\rangle \hat{F}^{(\ell)} +  \left\langle \cdot, \hat{F}^{(\ell)} \right\rangle F^{(\ell)} \right) + \smallO(\e). \label{eq:cristina}
\end{align}

Finally, we prove $(iii)$. Note that $(i)$ implies $\langle F^{(\ell)}, \hat{F}^{(\ell)} \rangle = 0$. From the construction of the operator $\smallO(\e)$ in \eqref{eq:cristina}, it is clear that this operator is of finite rank and satisfies $\|\smallO(\e)\|/\e \xrightarrow[]{\e \to 0} 0.$ 

\end{proof}

\section{Proof of Theorem \ref{thm:2} \texorpdfstring{$2(c)$}{2(c)} }

From Lemma \ref{prop:xama} and Theorem \ref{thm:3} we obtain that for every $\ell\in\{1,\ldots,N\}$ and for every $\e>0$ sufficiently small
\begin{align}
    \lambda_\e^{(\ell)}= e^{-2\pi i k \beta_s} + \e \hat\lambda^{(\ell)} + \smallO(\e)\ \text{for some }s\in\{1,\ldots,S\}.  \label{eq:star}
\end{align}
In this section we show that there exists $\doublehat{\lambda}^{(\ell)}\in\mathbb C$ such that
\begin{align}
  \lambda_\e^{(\ell)}= e^{-2\pi i k \beta_s} + \e \hat\lambda^{(\ell)} + \e^2 \doublehat{\lambda}^{(\ell)}+ \smallO(\e^2)\ \text{for some }s\in\{1,\ldots,S\}.\label{eq:lambdahhat1}  
\end{align}
The proof of Theorem \ref{thm:2} is complete once the above equation is established. We first establish the following lemma before proving equation \eqref{eq:lambdahhat1}.

\begin{lemma}
  Let $\gamma>0$ be as in Proposition \ref{prop:nicefam}. For every $\ell\in\{1,\ldots, N\}$ the map 
  \begin{align}
      \e\in [0,\gamma)\mapsto \begin{cases}
        \displaystyle \dot{\lambda}_{\e}^{(\ell)} = \frac{\partial}{\partial \e} \lambda_{\e}^{(\ell)},& \text{if}\ \e \in (0,\gamma_k),\vspace{0.1cm}\\
        \hat{\lambda}^{(\ell)},&\text{if}\ \e=0,
    \end{cases} \label{eq:map}
  \end{align}
is continuous at $0$ and differentiable in $(0,\gamma).$\label{lem:lambcont}
\end{lemma}
\begin{proof}
   Let $\ell\in\{1,\ldots,N\}$ and $s\in\{1,\ldots,S\}$ such that $\ell \in B_{s}$.
 Since $P_{\e,\beta} \displaystyle f^{(\ell)}_{\e} = \lambda_{\e}^{(\ell)} \displaystyle f^{(\ell)}_{\e}$, we obtain that
    $$ \dot{P}_{\e,\beta} \displaystyle f^{(\ell)}_{\e} +P_{\e,\beta} \dot{f}^{(\ell)}_{\e}  = \dot{\lambda}_{\e}^{(\ell)} \displaystyle f^{(\ell)}_{\e} + \lambda_{\e}^{(\ell)} \dot{f}^{(\ell)}_{\e}.$$
By applying $\langle\cdot, \displaystyle \mu^{(\ell)}_{\e}\rangle$ in both sides, we obtain from \eqref{eq:mu1} and  $\displaystyle P_{\e,\beta}^*\mu_\e^{(\ell)} = \lambda_\e \mu_\e^{(\ell)}$ that
\begin{align}
    \dot{\lambda}_{\e}^{(\ell)} = \langle \dot{P}_{\e,\beta} \displaystyle f^{(\ell)}_{\e} ,\mu_\e^{(\ell)} \rangle = \langle D_\beta \dot{W} \displaystyle f^{(\ell)}_{\e} ,\mu_\e^{(\ell)} \rangle.\label{eq:lambdadot}
\end{align}

 Since $\e \mapsto \dot{\lambda}_\e^{(\ell)}$ is differentiable in $(0,\gamma)$ (see Proposition \ref{prop:nicefam}) to conclude the proof we must prove that \eqref{eq:map} is continuous at $0$. Therefore, from \eqref{eq:star} it is enough to show that
$$\lim_{\e\to 0} \dot{\lambda}_\e = \lim_{\e \to 0}\langle \dot{P}_{\e,\beta} \displaystyle f^{(\ell)}_{\e} ,\mu_\e^{(\ell)} \rangle = \lim_{\e \to 0}\frac{\lambda_{\e}^{(\ell)} - e^{-2\pi i k \beta_s}}{\e} = \hat{\lambda}^{(\ell)}.$$
Observe that for every $s\in \{1,\ldots,S\},$ 
\begin{align}
D_\beta - e^{-2 i k \beta_s} = (D_\beta - e^{-2 i k \beta_s})(1-\pi_s)\label{eq:pis}
\end{align} where $\pi_s: \C^N\to V_s$ is the $\mathbb C^N$-orthogonal projection onto $V_s$. From \eqref{eq:Ff1}, we obtain that
\begin{align*}
\langle \dot{P}_{\e,\beta} \displaystyle f^{(\ell)}_{\e} ,\mu_\e^{(\ell)} \rangle &= \frac{1}{\e}\left\langle \lambda_{\e}^{(\ell)} \displaystyle f^{(\ell)}_{\e} - D_\beta \displaystyle f^{(\ell)}_{\e},\mu_\e^{(\ell)}\right \rangle\\
   &= \frac{1}{\e} \left(\lambda_{\e}^{(\ell)} - e^{-2\pi i k\beta_s}\right) -\frac{1}{\e}  \left\langle \left(D_\beta -e^{-2\pi i k \beta_s}\right) f_\e^{(\ell)} , \mu^{(\ell)}_\e \right\rangle,\\
   &= \frac{1}{\e} \left(\lambda_{\e}^{(\ell)} - e^{-2\pi i k\beta_s}\right) -\frac{1}{\e} \left\langle \left(D_\beta -e^{-2\pi i k \beta_s}\right) (1-\pi_s) f_\e^{(\ell)} , \mu^{(\ell)}_\e \right\rangle.
\end{align*} 
  Since $[\mu_\e^{(\ell)}]\xrightarrow[]{\e\to 0} [f^{(\ell)}]$ with $f^{(\ell)} \in V_s$ and  $(1 - \pi_s) \displaystyle f_\e^{(\ell)}/\e$ converges to an element in the orthogonal complement of $V^s$ as $\e\to 0$ (see \ref{prop:linresponselast}),  we have that
$$\lim_{\e\to 0}\left\langle \frac{1}{\e} (1-\pi_s) f_\e^{(\ell)} , \mu^{(\ell)}_\e \right\rangle =0, $$
which implies that
$$\lim_{\e\to 0} \dot{\lambda}_\e^{(\ell)}= \lim_{\e \to 0}\langle \dot{P}_{\e,\beta} \displaystyle f^{(\ell)}_{\e} ,\mu_\e^{(\ell)} \rangle = \lim_{\e \to 0}\frac{\lambda_{\e}^{(\ell)} - e^{-2\pi i k \beta_s}}{\e} = \hat{\lambda}^{(\ell)}.$$

\end{proof}

\begin{theorem}\label{thm:hlambda}
    Given $\ell\in\{1,\ldots, N\}$, we have that
    $$\lambda_{\e}^{(\ell)} = e^{-2\pi i k \beta_{s_\ell}} + \e \hat{\lambda}^{(\ell)} + \e^2 \doublehat{\lambda}^{(\ell)} + \smallO(\e^2), $$
where $s_\ell \in\{1,\ldots,S\}$ such that $\ell \in B_{s_\ell},$ $\hat{\lambda}^{(\ell)}$ is given by Theorem \ref{thm:3} and 
$$\doublehat{\lambda}^{(\ell)}= 2\sum_{s=1,s\neq s_\ell}^{S} \frac{1}{e^{-2\pi i k \beta_{s_\ell}} - e^{-2\pi i k \beta_s}} \left\langle 
 \pi_s D_\beta \dot{W} f^{(\ell)} ,\pi_s  \dot{W} D_{-\beta} f^{(\ell)}  \right\rangle$$
and  $\pi_s:\C^N\to V_s$ is the $\mathbb C^N$-orthogonal projection onto $V_s$.

\end{theorem}

\begin{proof}
Recall from \eqref{eq:lambdadot} that $\dot{\lambda}^{(\ell)}_{\e}  = \langle \dot{P}_{\e,\beta} \displaystyle f^{(\ell)}_{\e},\mu_\e^{(\ell)}  \rangle =  \langle D_\beta \dot{W} \displaystyle f^{(\ell)}_{\e},\mu_\e^{(\ell)}  \rangle$
for every $\e>0$ small enough. Hence
$$\ddot{\lambda}^{(\ell)}_{\e}  = \langle D_\beta \dot{W} \dot{f}^{(\ell)}_{\e},\mu_\e^{(\ell)}\rangle + \langle D_\beta \dot{W} \displaystyle f^{(\ell)}_{\e},\dot{\mu}_\e^{(\ell)}\rangle =  \langle D_\beta \dot{W} \displaystyle f^{(\ell)}_{\e},\dot{\mu}_\e^{(\ell)}\rangle+  \langle \dot{f}^{(\ell)}_{\e}, \left(D_\beta \dot{W} \right)^*\mu_\e^{(\ell)}\rangle.$$
From \eqref{eq:Ff1}, \eqref{eq:pis} and
$$\langle \dot{f}^{(\ell)}_{\e},\mu_\e^{(\ell)}\rangle +  \langle \displaystyle f^{(\ell)}_{\e},\dot{\mu}_\e^{(\ell)}\rangle =\frac{\partial}{\partial \e} \langle f_\e^{(\ell)},\mu_\e^{(\ell)} \rangle = \frac{\partial}{\partial \e }1  =  0,  $$
we obtain that
\begin{align*}
\ddot{\lambda}^{(\ell)}_{\e} 
=&\ \left\langle\frac{1}{\e}\left(\lambda_{\e}^{(\ell)} f^{(\ell)}_{\e} - D_\beta f^{(\ell)}_{\e}\right),\dot{\mu}^{(\ell)}_\e \right\rangle 
+ \left\langle \dot{f}_\e^{(\ell)}, \frac{1}{\e}\left(\overline{\lambda_{\e}^{(\ell)}} \mu^{(\ell)}_{\e} - \overline{D_\beta} \mu^{(\ell)}_{\e}\right)\right\rangle \\
=&\ \frac{1}{\e} \lambda_{\e}^{(\ell)} \left( \langle \dot{f}^{(\ell)}_{\e}, \mu_\e^{(\ell)} \rangle + \langle f^{(\ell)}_{\e}, \dot{\mu}_\e^{(\ell)} \rangle \right)  
- \left\langle D_\beta \dot{f}^{(\ell)}_{\e}, \frac{1}{\e} \mu_\e^{(\ell)} \right\rangle 
- \left\langle \frac{1}{\e} D_\beta f^{(\ell)}_{\e}, \dot{\mu}_\e^{(\ell)} \right\rangle \\
=&\ \left\langle -e^{-2\pi i k\beta_{s_\ell}} \dot{f}^{(\ell)}_{\e} + e^{-2\pi i k\beta_{s_\ell}} \dot{f}^{(\ell)}_{\e} - D_\beta \dot{f}^{(\ell)}_{\e}, \frac{1}{\e} \mu_\e^{(\ell)} \right\rangle \\
&\quad + \left\langle \frac{1}{\e}\left( -e^{-2\pi i k\beta_{s_\ell}} f^{(\ell)}_{\e} + e^{-2\pi i k\beta_{s_\ell}} f^{(\ell)}_{\e} - D_\beta f^{(\ell)}_{\e} \right), \dot{\mu}_\e^{(\ell)} \right\rangle \\
=&\ \frac{-e^{-2\pi i k\beta_{s_\ell}}}{\e} \left( \left\langle \dot{f}^{(\ell)}_{\e}, \mu_\e^{(\ell)} \right\rangle + \left\langle f^{(\ell)}_{\e}, \dot{\mu}_\e^{(\ell)} \right\rangle \right) + \left\langle (e^{-2\pi i k \beta_{s_\ell}} - D_\beta)(1 - \pi_s)\dot{f}_\e^{(\ell)}, \frac{1}{\e} \mu_\e^{(\ell)} \right\rangle \\
&\quad + \left\langle \frac{1}{\e}(e^{-2\pi i k \beta_{s_\ell}} - D_\beta)(1 - \pi_{s_\ell}) f_\e^{(\ell)}, \dot{\mu}_\e^{(\ell)} \right\rangle \\
=&\ \left\langle (e^{-2\pi i k \beta_{s_\ell}} - D_\beta)(1 - \pi_s)\dot{f}_\e^{(\ell)}, \frac{1}{\e} (1 - \pi_s)\mu_\e^{(\ell)} \right\rangle \\
&\quad + \left\langle \frac{1}{\e}(e^{-2\pi i k \beta_{s_\ell}} - D_\beta)(1 - \pi_{s_\ell}) f_\e^{(\ell)}, (1 - \pi_{s_\ell})\dot{\mu}_\e^{(\ell)} \right\rangle.
\end{align*}
where the last equality follows from the fact that $(1-\pi_s)$ is an orthogonal projection. From Propositions  \ref{prop:linresponselast} and \ref{prop:linresponselastmu} we obtain that
\begin{align*}
    \lim_{\e\to 0} \ddot{\lambda}_\e^{(\ell)} =&\ \left\langle (e^{-2\pi i k \beta_{s_\ell}} - D_\beta)\lim_{\e\to 0}(1 - \pi_s)\dot{f}_\e^{(\ell)}, \lim_{\e\to 0} \frac{1}{\e} (1 - \pi_s)\mu_\e^{(\ell)} \right\rangle \\
&\quad + \left\langle (e^{-2\pi i k \beta_{s_\ell}} - D_\beta)\lim_{\e\to 0}\frac{1}{\e}(1 - \pi_{s_\ell}) f_\e^{(\ell)},\lim_{\e\to 0} (1 - \pi_{s_\ell})\dot{\mu}_\e^{(\ell)} \right\rangle.\\
=&\sum_{s=1,s\neq s_\ell}^{S} (e^{-2\pi i k \beta_{s_\ell}} - e^{-2 \pi i k \beta_s}) \sum_{j\in B_s} \left(\lim_{\e\to 0}\dot{f}_{\e}\right)_j \overline{\left( \lim_{\e\to 0} \frac{1}{\e}\mu_\e^{(\ell)} \right)_j}  \\
&\quad +\sum_{s=1,s\neq s_\ell}^{S} (e^{-2\pi i k \beta_{s_\ell}} - e^{-2 \pi i k \beta_s}) \sum_{j\in B_s} \left(\lim_{\e\to 0} \frac{1}{\e}f_{\e}\right)_j \overline{\left( \lim_{\e\to 0} \dot{\mu}_\e^{(\ell)} \right)_j}  \\
    =&2\sum_{s=1,s\neq s_\ell}^{S} \frac{1}{e^{-2\pi i k \beta_{s_\ell}} - e^{-2\pi i k \beta_{s}}} \sum_{j\in B_{s}} \left(D_\beta \dot W f^{(\ell)}\right)_j \left(\overline{\dot W D_{-\beta}f^{(\ell)}}\right)_j
\end{align*}

From Lemma \ref{lem:lambcont} the map $$\e\in[0,\gamma)\mapsto \begin{cases}
        \displaystyle \dot{\lambda}_\e^{(\ell)} = \frac{\partial}{\partial \e} \lambda_{\e}^{(\ell)},& \text{if}\ \e\in(0,\gamma),\vspace{0.1cm}\\
        \hat{\lambda}^{(\ell)},&\text{if}\ \e=0,
    \end{cases}$$ is continuous. From the Mean Value Theorem we have that for every $\e>0$, sufficiently small, there exists $c_\e \in (0,\e)$ such that
    $$ \dot{\lambda}_{\e}^{(\ell)} -\hat{\lambda}^{(\ell)} = \ddot{\lambda}_{c_\e} \e,$$
therefore
\begin{align*}
    \lim_{\e\to 0}\frac{\dot{\lambda}_{\e}^{(\ell)} -\hat{\lambda}^{(\ell)}}{\e} &=2  \sum_{s=1,s\neq s_\ell}^{S} \frac{1}{e^{-2\pi i k \beta_{s_\ell}} - e^{-2\pi i k \beta_s}} \left\langle 
 \pi_s D_\beta \dot{W} f^{(\ell)} ,\pi_s  \dot{W} D_{-\beta} f^{(\ell)}  \right\rangle
\end{align*}
\end{proof}

Below, we finally prove Theorem \ref{thm:2}. 

\begin{proof}
    The proof of Theorem \ref{thm:2} follows from a direct combination of Proposition \ref{prop:3.3} with Theorems \ref{thm:3}, \ref{thm:fhat} and \ref{thm:hlambda}.
\end{proof}

\section{Proof of Proposition \ref{prop:LRA}}
\label{sec:alpharesp}

The proof of Proposition \ref{prop:LRA} follows from standard techniques in perturbation theory (see \cite{Kato}).

\begin{proof}[Proof of Proposition \ref{prop:LRA}]
From Lemma \ref{lem:eform} and \cite[II.2.2]{Kato}, we obtain that the maps in \eqref{pro:sard} are analytic functions. Equation \eqref{eq:scarpeli} follows from differentiating the identity $D_{k,\alpha} W_\e f_{\alpha,k,\e} = \lambda_{k,\e,\alpha} f_{k,\e,\alpha},$
with respect to $\alpha$ and observing, we can invert the left-hand side of the identity
$$ \left(D_{k,\alpha} W_\e - \lambda_{k,\e,\alpha} \right) \frac{\partial}{\partial \alpha}f_{k,\e,\alpha} = \left(D_{k,\alpha} W_\e - \lambda_{k,\e,\alpha} \right)^{-1} \left(\frac{\partial }{\partial \alpha}D_{k,\alpha} W_\e - \frac{\partial}{\partial \alpha} \lambda_{k,\e,\alpha}\right) f_{k,\e,\alpha},$$
since $\langle f_{k,\e,\alpha}, \frac{\partial}{\partial \alpha^\ell} f_{k,\e,\alpha} \rangle = 0,$ for each $\ell\in\{1,\ldots,N\}.$

\end{proof}

\section[Explicit formulae for $\hat{\lambda}_{k,\beta}^{(\ell)}$ and $f^{(\ell)}$ for $W_\e$ (see (1))]
{Explicit formulae for \texorpdfstring{$\hat{\lambda}_{k,\beta}^{(\ell)}$}{lambda} and \texorpdfstring{$f^{(\ell)}$}{f} for \texorpdfstring{$W_\e$}{W} defined in \texorpdfstring{$(1)$}{(1)}}

\begin{proposition}[\cite{Losonczi1992,Eigen}]
\label{eigen}
Let $\hat{P}_{\beta}$ be the matrix defined in Theorem \ref{thm:3}. The eigenvalues of $\hat{P}_{\beta}$ are
$$\hat\lambda_{k,\beta}^{(\ell)} = \begin{cases}
 \displaystyle   e^{-2 \pi i k \beta_1}\left(-1 +   \cos \left(\frac{2 \ell -1}{2N_1+1}\pi \right)\right), &\text{if }\ell\in\{1,2,\ldots, N_1\}\vspace{0.1cm}\\
 \displaystyle  e^{-2 \pi i k \beta_s} \left(-1 +   \cos \left(\frac{ \ell-N_{s-1}}{N_{s}-N_{s-1} +1} \pi \right)\right), &\text{if }\ell\in\{N_{s-1}+1,\ldots, N_s\},s\neq 1,S\\
   \displaystyle  e^{-2 \pi i k \beta_S} \left(-1 +   \cos \left(\frac{2 (\ell-N_{S-1}) -1}{2(N-N_{S-1}) +1}\pi \right)\right), &\text{if }\mmc{\ell}\in\{N_{S-1}+1,\ldots, N\}
\end{cases} $$
and their respective eigenvectors $f^{(j)}$  are:
\begin{enumerate}
    \item[$(a)$] if $\ell\in\{1,\ldots,N_1\}$
    $$\left(f^{(\ell)}\right)_j = \begin{cases}
     \displaystyle  \sin \left( \pi (2 \ell+1) \left( \frac{N_1+1}{2 N_1+1} -  \frac{j}{2 N_1+1} \right)  \right) , &\text{if }j\in\{1,\ldots, N_1\}\\
     0, &\text{if }j\in\{N_1+1,\ldots, N\}
\end{cases}.$$
\item[$(b)$]  if $\ell\in\{N_{s-1}+1,\ldots,N_s\}$ and $s\in\{2,\ldots,S-1\}$
    $$\left(f^{(\ell)}\right)_j = \begin{cases}
     \displaystyle  \sin \left( \pi \frac{ (\mmc{\ell} - N_{s-1})(j-N_{s-1})   }{N_{s}-N_{s-1} +1}  \right) , &\text{if }j\in\{N_{s-1}+1,\ldots, N_s\}\\
     0, &\text{if }j\in \{1,\ldots,N\}\setminus \{N_{s-1}+1,\ldots, N_s\}
\end{cases}.$$
\item[$(c)$] if $\ell\in\{N_{S-1}+1,\ldots,N\}$ 
$$\left(f^{(\ell)}\right)_j=  \begin{cases}
    0 , &\text{if }j\in\{1,2,\ldots, N_{S-1}\}\\
      \displaystyle  \sin \left( \pi (2 \ell+1)  \frac{j-(N-N_{S-1})}{2 (N-N_{S-1})+1}  \right) , &\text{if }j\in\{N_{S-1}+1,\ldots, N\}
\end{cases}.  $$
\end{enumerate}
\end{proposition}
The quantities $\doublehat{\lambda}_k$ and $\hat{f}_k$ can be explicitly computed using the above equations and Theorem \ref{thm:2} (via Theorems \ref{thm:fhat} and \ref{thm:hlambda}).

\section{Acknowledgements}
GF acknowledges productive early discussions with Dimitris Giannakis about a related model during a visit to Dartmouth in 2023.
GF and MMC are supported by an Australian Research Council Laureate Fellowship (FL230100088).


\bibliographystyle{plain}
\bibliography{refs}

\newpage

\end{document}